\documentclass[11pt, twoside, a4paper]{article}
\usepackage[utf8]{inputenc}
\usepackage{amsmath, amssymb, amsthm} 
\usepackage{amstext, amsfonts, a4}
\usepackage[english]{babel}
\usepackage{xcolor}
\usepackage{bbm}
\usepackage{enumerate} 

\allowdisplaybreaks

\usepackage[margin=2.0cm]{geometry}
\usepackage{dsfont}

\usepackage[hidelinks]{hyperref}
\hypersetup{
 colorlinks,
 linkcolor={red!50!black},
 citecolor={blue!50!black},
 urlcolor={blue!80!black}
}

\numberwithin{equation}{section}

\newcommand{\Ind}{\mathds{1}}

\newcommand{\ntriple}[1]{{\left\vert\kern-0.25ex\left\vert\kern-0.25ex\left\vert #1 
    \right\vert\kern-0.25ex\right\vert\kern-0.25ex\right\vert}}

\newcommand{\der}[2]{\frac{\dd #1}{\dd #2}}

\newcommand{\dd}{\mathrm{d}}
\newcommand{\Diff}{\mathrm{D}}

\newcommand{\supp}{\operatorname{supp}}

\newcommand{\curl}{\operatorname{curl}}
\renewcommand{\div}{\operatorname{div}}

\newcommand{\dist}{\operatorname{dist}}

\renewcommand{\leq}{\leqslant}\renewcommand{\le}{\leqslant}
\renewcommand{\geq}{\geqslant}\renewcommand{\ge}{\geqslant}

\renewcommand{\epsilon}{\varepsilon}

\newcommand{\eps}{\varepsilon}

\newcommand{\R}{\mathbb{R}}

\newcommand{\ds}{\displaystyle}

\renewcommand{\sp}{\hspace{0.2cm}}

\renewcommand{\tilde}{\widetilde}

\theoremstyle{plain}

\newtheorem{theo}{Theorem}[section]
\newtheorem{lemme}[theo]{Lemma}
\newtheorem{prop}[theo]{Proposition}
\newtheorem{coro}[theo]{Corollary}
\newtheorem{defin}[theo]{Definition}
\newtheorem{rem}[theo]{Remark}

\def\cal#1{\mathcal{#1}}

\title{Dynamics of helical vortex filaments in non viscous incompressible flows}
\author{Martin Donati, Christophe Lacave, Evelyne Miot}
\def\adrese{
\noindent M. Donati: Université de Lyon, ENS de Lyon, Unité de Mathématiques Pures
et Appliquées, 69364 Lyon, France. \\
C. Lacave: Univ. Savoie Mont Blanc, CNRS, LAMA, ISTerre, 73000 Chamb\'ery, France.\\
E. Miot: Univ. Grenoble Alpes, CNRS, IF, 38000 Grenoble, France.\\
{\it Email address:}\\
\texttt{Martin.Donati@univ-grenoble-alpes.fr}\\
\texttt{Christophe.Lacave@univ-smb.fr}\\
\texttt{Evelyne.Miot@univ-grenoble-alpes.fr}
}

\date{\today}

\begin{document}
\maketitle

\begin{abstract}
 In this paper we study concentrated solutions of the three-dimensional Euler equations in helical symmetry without swirl. We prove that any helical vorticity solution initially concentrated around helices of pairwise distinct radii remains concentrated close to filaments. As suggested by the vortex filament conjecture, we prove that those filaments are translating and rotating helices. Similarly to what is obtained in other frameworks, the localization is weak in the direction of the movement but strong in its normal direction, and holds on an arbitrary long time interval in the naturally rescaled time scale. In order to prove this result, we derive a new explicit formula for the singular part of the Biot-Savart kernel in a two-dimensional reformulation of the problem. This allows us to obtain an appropriate decomposition of the velocity field to reproduce recent methods used to describe the dynamics of vortex rings or point-vortices for the lake equation.
\end{abstract}

\section{Introduction}

The purpose of this paper is to study the time evolution for 3D inviscid flows for which the vorticity is initially concentrated around helical curves.

We consider the Euler equations governing the dynamics of a three-dimensional inviscid, incompressible fluid in a domain $\Omega$:
\begin{equation}\label{eq:Euler3D}\tag{E}
\begin{cases}
 \ds \partial_t U + (U \cdot \nabla) U = -\nabla P & \text{ in } \Omega \times \R_+^* ,\\
 \div(U)=0& \text{ in } \Omega \times \R_+ ,\\
 U\cdot n =0 & \text{ on } \partial\Omega \times \R_+,
\end{cases}
\end{equation}
where $n$ is the outward normal vector, $U: \Omega \times \R_+^* \to \R^3$ denotes the velocity of the fluid and $P$ the pressure. We shall focus on particular flows, called vortex filaments, for which the vorticity $\curl(U)$ is sharply concentrated in a thin tube around a curve in $\R^3$. Understanding the stability (namely, whether the concentration of the filament around a curve persists in time) and the dynamics of vortex filaments in three-dimensional flows are a longstanding issue in mathematical physics. Da~Rios formally derived in \cite{DaRios} that, to leading order, the asymptotic motion law for one single vortex filament in a tube of size $\eps$ around a curve parametrized by $\chi(\cdot,t)$, with arc-length parameter $\sigma$, is governed by the binormal curvature flow:
\begin{equation}\tag{BF}
 \label{eq:BF} \partial_t \chi = c |\ln \eps| (\partial_\sigma \chi \times \partial_{\sigma \sigma} \chi)
\end{equation}
where $c$ is the curvature. Note that \eqref{eq:BF} exhibits some trivial solutions: the stationary vortex line, the uniformly translating circle (known as ``vortex ring'') and the translating-rotating helix (referred to as ``helical filaments'' in the remaining of the paper). We refer e.g. to \cite{majda-bertozzi} and to references therein for a general introduction on the subject. 
The ``vortex filament conjecture'' is the conjecture that vorticity initially concentrated around a curve remains close to a curve evolving according to \eqref{eq:BF} to leading order for at least a certain interval of time. While it is completely settled in the 2D case (where vortex filaments reduce to point vortices), see \cite{MarPul93}, it is open in general. Jerrard and Seis \cite{JerrardSeis} provided a rigorous derivation of \eqref{eq:BF} \emph{assuming} the vorticity remains concentrated around the curve.

Without assuming \emph{a priori} concentration, further results have been obtained under supplementary symmetry assumptions. For axisymmetric flows without swirl, Butt\`a, Cavallaro and Marchioro \cite{Mar3} recently rigorously justified the dynamics of several vortex rings of different radii. They also established a ``semi-strong'' localization result: the filaments remains for all time sharply localized in the radial direction (namely with respect to the distance to the symmetry axis). Their approach inspired a recent work by Hientzsch, Lacave and Miot \cite{Hientzsch_Lacave_Miot_2022_Dynamics_of_PV_for_the_lake_eq} in the setting of point vortices for the lake equations, which is a 2D model for incompressible flows inheriting an anelastic constraint from the 3D case.

In the special case of vortex rings, Fraenkel \cite{Fraenkel} exhibited a family of solutions of \eqref{eq:Euler3D} such that the corresponding vorticity concentrates for all time on a curve solution of \eqref{eq:BF}. 
In \cite{DavilaPinoMussoWei22} Davila, Pino, Musso and Wei constructed such a family of solutions that do not change form, concentrating to one or several polygonally distributed rotating-translating helical filaments, by means of elliptic singular perturbation techniques. In a series of papers \cite{Cao_Wan_2023,Cao_Wan_2023bis,Cao_Wan_2023ter,Cao_Wan_2022_Structure_of_Green,Cao24,Cao_Li_Qin_Wan_2025}, several desingularization results were obtained for helical vortex filaments by means of various variational techniques, enabling in particular to construct helical vortex patches. In particular, \cite{Cao_Wan_2023} obtains helical vortices with compact support, \cite{Cao_Wan_2022_Structure_of_Green} obtains helical vortex patches and \cite{Cao24} covers both cases and some other profiles by use of rearrangements techniques. We also mention the recent work by Guerra and Musso~\cite{GuerraMusso}, that constructs a special family of solutions concentrating to a collapsing configuration of helical filaments.
We emphasize the fact that in all the works mentioned above, the solutions always belong to a restricted class of functions. Indeed, the approach is to construct a specific class of solutions with a prescribed vorticity profile, except in \cite{Cao24} where a general initial vortex profile is chosen but the solution for all time is then constructed by means of a rearrangement of this profile, restricting the actual choice of the initial data.

\medskip

Our objective here is to establish the dynamics of helical filaments starting from generic initial data, with very few and natural assumptions relating only to the initial concentration, and with other techniques.

\medskip

The helical symmetry is a physically relevant framework since this symmetry is obtained in many different contexts, in particular in the wake of rotors. This covers a wide range of situations from the study of wind or water turbine to vertical flight of helicopters for instance. Moreover, in the wake of each wing of an airplane, vortices of the same sign are created on straight lines, but immediately start interacting with each other, inducing a rotation that creates a local helical symmetry. Recent papers (see for instance \cite{Bolnot_LD_Leweke_Pairing_Instability,blanco-rodriguez_ledizes_selcuk_delbende_rossi_2015,castillo-castellanos_ledizes_2022}) study theoretically, numerically and experimentally physical properties of such flows and in particular instabilities due to non helical perturbations and viscosity.

\medskip

We now introduce with some more details our working framework. We focus on flows with helical symmetry and without helical swirl. More precisely, following \cite{Dutrifoy,EttingerTiti}, for some fixed $h>0$ we define the following operators
for all $\theta\in \R$:
 \begin{equation*}
 R_\theta = \begin{pmatrix}
 \cos \theta & -\sin\theta & 0 \\
 \sin\theta&\cos\theta &0\\
 0&0&1
 \end{pmatrix}
 \quad \text{and}\quad
 S_{\theta,h}x=R_\theta x + h\begin{pmatrix}
 0\\0\\\theta 
 \end{pmatrix},\quad x\in \Omega
 \end{equation*}
and we say that $\Omega$ is a helical domain if it satisfies $ S_{\theta,h}\Omega=\Omega$, for all $\theta \in \R$.
We say that
 $(U,P)$ is a helical solution to \eqref{eq:Euler3D} on the helical domain $\Omega$ if:
 \begin{equation*}
 U(S_{\theta,h} x)=R_\theta U(x), \quad P(S_{\theta,h} x)=P(x),\quad \forall x\in \Omega, \quad \forall \theta \in \R.
 \end{equation*}
Finally, we say that $(U,P)$ is helical without swirl if $(U,P)$ is helical and $U$ is orthogonal to the helices, namely:
 \begin{equation}\label{def:xi}
 U(x)\cdot \xi(x)=0, \quad \forall x=\begin{pmatrix} x_1\\x_2\\x_3\end{pmatrix}\in \Omega,
 \text{ where }
 \xi(x)=\begin{pmatrix} -x_2\\ x_1 \\ h\end{pmatrix}.
 \end{equation}
 Such properties are formally preserved by \eqref{eq:Euler3D}.
 
 Global existence and uniqueness of weak solutions to \eqref{eq:Euler3D} that are helical without swirl have been proved in \cite{Dutrifoy, EttingerTiti, BronziLopes,JiuLiNiu17,GuoZhao23}. Such solutions are also Lagrangian, see Section~\ref{subsec:2D} hereafter for more details. It turns out that the ``no swirl'' condition implies that the vorticity is parallel to $\xi$ which allows us to define a scalar quantity $\omega$:
 \begin{equation}\label{def:omega}
\curl U(x,t) = \frac1h \omega( \tilde R_{-\frac{x_3}{h}}(x_1,x_2),t) \xi(x),
\text{ where }
 \tilde R_\theta = \begin{pmatrix}
 \cos \theta & -\sin\theta \\
 \sin\theta&\cos\theta
 \end{pmatrix}.
\end{equation}
 From this observation, it was proved in \cite{EttingerTiti} that, for helical flows without swirl, Equation \eqref{eq:Euler3D} reduces to a two-dimensional system for the vorticity posed on the 2D cross-section $\mathcal{U}=\{(x_1,x_2)|(x_1,x_2,0)\in \Omega\}$ of $\Omega$:
\begin{equation}\label{eq:Helico2D}
\begin{cases}
 \partial_t \, \omega + v\cdot \nabla \omega = 0 & \text{ in } \mathcal{U}\times \R_+^* , \vspace{1mm}\\
 v = \nabla^\perp \Psi & \text{ in } \mathcal{U} \times \R_+ ,\vspace{1mm} \\
 \div \big( K\nabla \Psi \big) = \omega & \text{ in } \mathcal{U} \times \R_+ ,\quad \Psi=0\text{ on } \partial \mathcal{U} \times \R_+ , \vspace{1mm} \\
 \omega(\cdot,0) = \omega_0 & \text{ in } \mathcal{U},
\end{cases}
\end{equation} 
where $K$ is a symmetric positive-definite matrix defined by
\begin{equation}\label{def:K}
 K(x) = \frac{1}{x_1^2+x_2^2+h^2} \begin{pmatrix} h^2 + x_2^2 & -x_1x_2 \\ -x_1x_2 & h^2 + x_1^2 \end{pmatrix}.
\end{equation}
We have used the notation, $\nabla^\perp \Psi = (\nabla\Psi)^\perp$ with the convention $(a,b)^\perp=(-b,a)$. For this result and for the rest of this article, we assume that $\cal{U}$ is a bounded, simply connected domain, with $C^{1,1}$ boundary for simplicity reasons. More details on weak solutions to \eqref{eq:Helico2D} will be given in Section~\ref{subsec:2D}.
 
In the 2D reduction of the 3D system, vortex filaments reduce to point vortices, that correspond to the 2D projection of the filaments. So we are left to investigating the persistence and dynamics of point vortices for Equation \eqref{eq:Helico2D}. 
Note that in view of \eqref{eq:BF}, it is more judicious to consider another time-scale in order to obtain a velocity of order one when considering concentrated vortices, thus we are led to consider a rescaled system:
\begin{equation}\label{eq:Helico2D_rescaled}
\begin{cases}
 \ds \partial_t \, \omega^\eps + \frac{1}{|\ln \eps|}v^\eps\cdot \nabla \omega^\eps = 0 & \text{ in } \mathcal{U} \times \R_+^*, \vspace{1mm}\\
 v^\eps = \nabla^\perp \Psi^\eps & \text{ in } \mathcal{U} \times \R_+ ,\vspace{1mm} \\
 \div \big( K \nabla \Psi^\eps \big) = \omega^\eps & \text{ in } \cal U\times\R_+ ,\quad \Psi^\epsilon=0\text{ on } \partial \mathcal{U} \times \R_+ , \vspace{1mm} \\
 \omega^\eps(\cdot,0) = \omega_0^\eps & \text{ in } \mathcal{U}.
\end{cases}
\end{equation}
One of the contributions of this paper is that in Proposition~\ref{prop:Biot_Savart} we obtain an important decomposition of the Green's function of the operator $\mathcal{L} = \div (K \nabla \cdot)$ as
\begin{equation*}
 \mathcal{G}_{K,\mathcal{U}} := G_K + S_{K,\mathcal{U}},
\end{equation*}
where $S_{K,\mathcal{U}} \in W^{1,\infty}$ and $G_K$ is explicitly given at \eqref{def:G_K-0}. A similar decomposition was obtained in \cite{Cao_Wan_2022_Structure_of_Green} exhibiting a different singular term, but with a remainder only locally H\"older. Here we obtain a more precise decomposition with the stronger regularity $W^{1,\infty}$ which we crucially need in the following. From this Green's function, we obtain a Biot-Savart law that gives in Proposition~\ref{prop:decomp_vitesse} a sharp decomposition of the velocity field $v^\eps$.

\medskip

We may now state our main result as follows:
\begin{theo}\label{theo:main} Assume that there exists $R_{\cal U}>0$ such that $B(0,R_{\cal U})\subset \mathcal{U}$.
Let $(z_{i,0})_{1 \le i \le N}$ be $N$ points in $B(0,R_{\cal U})$ such that $|z_{i,0}| \neq |z_{j,0}|$ for every $i \neq j$. Let $\gamma_i \in \R^*$.

For every $\eps >0$ such that $\eps< R_{\cal U}-\max_i |z_{i,0}|$, let $\omega_0^\eps \in L^\infty (\mathcal{U})$ such that
\begin{equation}\label{hyp:omega0}
 \left\{\begin{aligned}
& \omega_0^\eps = \sum_{i=1}^N \omega_{i,0}^\eps, \\
& \supp \omega_{i,0}^\eps \subset B\big( z_{i,0},\eps \big), \\
& \omega_{i,0}^\eps \text{ has a definite sign and } \int_{\mathcal{U}} \omega_{i,0}^\eps(x) \dd x = \gamma_i, \\
& |\omega_0^\eps| \le \frac{M_0}{\eps^2},\quad \text{for some $M_0>0$}.
\end{aligned}\right.
\end{equation}
For $T>0$, let $(v^\eps,\omega^\eps)$ be the unique weak solution of \eqref{eq:Helico2D_rescaled} on $[0,T]$ (in the sense of Definition~\ref{defi:weaksol2D} below).
Let $z_i(t) = \tilde{R}_{t \nu_i } z_{i,0}$ with
\begin{equation*}
 \nu_i = -\frac{\gamma_i}{4\pi h \sqrt{|z_{i,0}|^2+h^2}}.
\end{equation*}

Then, there exists a decomposition 
$$\omega^\eps=\sum_{i=1}^N \omega_i^\eps, \quad \omega_i^\eps \in L^\infty(\mathcal{U})$$ which satisfies:
\begin{itemize}
 \item[$(i)$] A weak localization property: there are $C_T, \eps_{T}>0$ such that, for any $\eps\in (0,\eps_{T}]$, we have 
 \begin{equation*}
 \sup_{t \in [0,T]}\left| \gamma_i - \int_{B(z_i(t),r_\eps)} \omega_i^\eps(x,t)\dd x\right| \le \frac{C_T}{\ln |\ln \eps|},\quad \text{where} 
 \quad r_\eps = \left(\frac{\ln |\ln\eps|}{|\ln \eps|}\right)^{1/2},
 \end{equation*}
 and
 \begin{equation*}
 \sup_{t \in [0,T]} \left|\frac{1}{\gamma_i}\int_{\mathcal{U}} x\omega_i^\eps(x,t)\dd x - z_i(t)\right| \le \frac{C_T}{\sqrt{|\ln\eps|}}.
 \end{equation*}
 
 \item[$(ii)$] A strong localization property in the radial direction: for every $\kappa \in (0,1/4)$, there is $C_{\kappa,T}$ and $\eps_{\kappa,T}>0$ such that, for every $\eps\in(0,\eps_{\kappa,T}]$, we have
 \begin{equation*}
 \supp \omega_i^\eps(\cdot,t) \subset \left\{ x \in \mathcal{U} \, , \, \Big||x| - |z_{i,0}|\Big| \le \frac{C_{\kappa,T}}{|\ln\eps|^\kappa} \right\}, \quad \text{for all}\quad t\in [0,T].
 \end{equation*}
\end{itemize}

\end{theo}

From Theorem~\ref{theo:main}, after reconstructing the 3D flow from the 2D solution, we indeed see that for any concentrated initial vorticity, the solution remains close to the translating rotating helices given by \eqref{eq:BF}. The velocities $\nu_i$ are coherent with this model and with other results on the subject (see \cite[Equation (3.1)]{Cao_Wan_2022_Structure_of_Green} for instance, where their $\nu$ has the opposite sign but their matrix of rotation is clockwise).

We recall that the main difference of Theorem~\ref{theo:main} with the results obtained in \cite{DavilaPinoMussoWei22} is that we prove the localization for a much larger class of initial data. However the localization in Theorem~\ref{theo:main} is only weak in the direction of the movement, while \cite{DavilaPinoMussoWei22} obtains a strong confinement in both directions due to the choice of a well prepared initial data. Due to this weak confinement, we are restricted to study helices of different radii which excludes in particular the case of polygonally distributed helices, covered in \cite{DavilaPinoMussoWei22}.

This weak confinement in the direction of the movement is a quite natural limitation. Indeed, for general initial data, even in the usual planar 2D case, some filamentation may happen in bounded time, namely some vorticity may be driven away from the core of the vortex. In 3D, this vorticity would then slow down compared to the core of the filament (recall that the leading order movement is due to the self interaction of the filament due to concentration and his curvature), which in turns drives the lost vorticity further away.

\medskip

Our strategy is to follow the techniques of \cite{Mar3, Mar2} and developed by \cite{Hientzsch_Lacave_Miot_2022_Dynamics_of_PV_for_the_lake_eq}. In particular, in order to lighten the computations, we obtain most estimates by focusing on a single filament and considering the influence of the other points as an exterior velocity field. This approach has already been described in detail in \cite{livre-jaune} and has been used in many works including \cite{Mar3,Mar2,Hientzsch_Lacave_Miot_2022_Dynamics_of_PV_for_the_lake_eq,MarPul93}. This is only a notation and the exterior field could be replaced by its real expression, given in \eqref{def:F}, at any time in the proof. 

The plan of the paper is the following. In Section~\ref{sec:helical}, we recall and establish properties of the reduction to the 2D problem. In particular, in Proposition~\ref{prop:Biot_Savart} we derive an explicit formula for the singular part of the Biot-Savart kernel for this problem. In Section~\ref{sec:reduction}, we set up the proof of Theorem~\ref{theo:main} and introduce the usual notations needed to focus on the motion of a single vortex filament in an exterior field. Section~\ref{sec:single_vortex} is dedicated to this reduced model. We control different quantities such as the energy and vorticity moments of the solution to derive the dynamics and obtain the localization results. Crucially, the estimates do not depend on which filament we focused on. Section~\ref{sec:conclusion} concludes the proof of Theorem~\ref{theo:main}, by coming back to the full expression of $N$ interacting filament.

\medskip

\emph{Remarks on notation.} Except in Subsection~\ref{subsec:2D}, $x$ will from now on denote a point $(x_1,x_2)$ in $\R^2$, and we save then the notation $X$ for the following definition: $|X| := \sqrt{|x|^2 + h^2}$. Similarly, $y$, $z$, and $b^\eps(t)$ will be points in $\R^2$ and $|Y|$, $|Z|$, and $|B^\eps(t)|$ should be understood according to the previous definition. Unless specified otherwise, the integrals are always on $\mathcal{U}$. The values of the constants named $C$ are always irrelevant and may change from line to line. The constants $C$ are allowed to depend on $\mathcal{U}$ and $h$ without further mention since those objects are fixed once and for all.

\section{The helical symmetry framework}\label{sec:helical}

\subsection{Some known facts on the 2D reduction and the well-posedness of the 3D Euler in helical symmetry}
\label{subsec:2D}

Let $\Omega$ be a helical domain with 2D cross-section $\mathcal {U}$. In order to use known results on well-posedness for the three-dimensional Euler equation \eqref{eq:Euler3D}, we shall always assume in this section that $\mathcal{U}$ is simply connected, bounded and has $C^{1,1}$ boundary. Let an initial vector field $U_0$ be smooth and without helical swirl (i.e. satisfying $(U_0\cdot \xi)(x)=0$, for all $x\in \Omega$ with $\xi$ defined in \eqref{def:xi}). It was proved in \cite{EttingerTiti} that any smooth helical solution of \eqref{eq:Euler3D} remains without helical swirl for positive times. In this case, the three-dimensional vorticity $\curl U$ is related to a scalar quantity $\omega$ via \eqref{def:omega}, and \eqref{eq:Euler3D} reduces to System~\eqref{eq:Helico2D} for $\omega$, with $K$ defined by \eqref{def:K}.

In domains with bounded cross-section, global well-posedness for smooth helical solutions of \eqref{eq:Euler3D} was obtained by Dutrifoy \cite{Dutrifoy}, whereas global well-posedness of weak (bounded) solutions of \eqref{eq:Helico2D} was proved by Ettinger and Titi \cite{EttingerTiti}. For the whole space $\R^3$, Bronzi, Lopes and Lopes \cite{BronziLopes} established the global existence of a weak solution for $\curl U_0\in L^p_c(\R^3)$. Such a result was extended in \cite{JiuLiNiu17} for $\curl U_0\in L^1\cap L^p(\R^3)$ (for $p\in (1,\infty]$).

Finally, Guo and Zhao used a Lagrangian method to establish in \cite{GuoZhao23} global existence and uniqueness of the weak helical solution without swirl for \eqref{eq:Euler3D} when the initial vorticity $\curl U_0\in L^1_1\cap L^\infty_1(\R^3)$. They also proved that the solution to the corresponding 2D reduction is a Lagrangian solution, namely it is constant along the characteristics of the flow associated to the velocity field.

\medskip

We shall consider here the following definition of a weak solution to
 \eqref{eq:Helico2D}, which is mainly inspired\footnote{The definition of weak solution in \cite{EttingerTiti} slightly differs from the one of the present paper, because it is given only in terms of $\Psi$, but it can be straightforwardly proved that it coincides with the one given below for a weak solution.} by \cite[Definition 3.10]{EttingerTiti}. Before stating our definition, we need to introduce the following operator
 \begin{equation} \label{def:L}
 \mathcal{L} \Psi := \div (K \nabla \Psi)
 \end{equation}
 for $\Psi$ an integrable function in $\mathcal{U}$.
\begin{defin} \label{defi:weaksol2D}
Let $\omega_0\in L^\infty(\mathcal{U})$. Let $\Psi_0\in W^{2,1}\cap H^{1}_0(\mathcal{U})$ be the unique solution of $\mathcal{L} \Psi_0=\omega_0$. We set $v_0=\nabla^\perp \Psi_0$. We say that $(v, \omega)$ is a weak bounded solution of \eqref{eq:Helico2D} on $[0,T]$ with initial condition $(v_0,\omega_0)$ if:
\begin{enumerate}
\item There exists $\Psi \in L^\infty([0,T], W^{2,1}(\mathcal{U}))$ with $\Psi=0$ a.e. on $\partial\mathcal{U}\times [0,T]$ such that
$v=\nabla^\perp \Psi$ and $\omega= \mathcal{L} \Psi$ a.e. in $\mathcal{U}\times [0,T]$;
\item We have $\omega \in L^\infty(\mathcal{U}\times [0,T])$;
\item For all test function $\Phi$ in $C^\infty_c(\mathcal{U}\times [0,T))$, we have
 \begin{equation*}
 - \int_{\mathcal{U}} \Phi(x,0) \omega_0(x) \,\dd x 
 = \int_{0}^T \int_{\mathcal{U}} \omega(x,s) \left(\partial_{t}\Phi + v\cdot \nabla \Phi\right)(x,s) \,\dd x\dd s.
\end{equation*}
\end{enumerate}
\end{defin}

As already mentioned, existence and uniqueness of the weak bounded solution as in Definition~\ref{defi:weaksol2D} for all $T>0$ is proved in \cite[Theorem 3.11]{EttingerTiti}. Moreover, the velocity field $v=\nabla^\perp\Psi$ satisfies the Calder\'on-Zygmund inequality \cite[Corollary 3.8]{EttingerTiti}: we have $\|\nabla v(\cdot,t)\|_{L^p}\leq C p \|\omega(\cdot,t)\|_{L^p}\leq C p\|\omega(\cdot,t)\|_{L^\infty}$ for all $2\leq p<\infty$. Observing that $\div v=0$ a.e. on $\mathcal{U}\times [0,T]$ and that $v\cdot n=0$ a.e. on $\partial \mathcal{U}\times [0,T]$, we may apply classical results by DiPerna and Lions \cite[p. 546]{diperna-lions} on the theory of linear transport equations and Lagrangian flows, see also \cite[Theorem 1, Theorem 2]{desjardins-96}. We infer that
there exists a unique measure-preserving Lagrangian flow $X:(x,t,t_0)\in \mathcal{U}\times [0,T]\times [0,T]\mapsto X(x,t,t_0)\in \mathcal{U}$ associated to $v$. Moreover, denoting further $X(x,t)=X(x,t,0)$ for simplicity,
the unique weak solution $\omega\in L^\infty(\mathcal{U}\times [0,T])$ to the transport equation $\partial_t \omega+v\cdot \nabla \omega=0$ satisfies \begin{equation*}
 \omega(\cdot,t)=X(\cdot,t)_\#\omega_0,\quad \forall t\in [0,T]
\end{equation*}
 in the sense that for all $\varphi\in C_c(\mathcal{U})$ we have $\int_{\mathcal{U}}\omega(x,t)\varphi(x)\,\dd x=\int_{\mathcal{U}}\omega_0(x)\varphi(X(x,t))\,\dd x$. 

By Morrey's inequality, Cald\'eron-Zygmund's estimate implies that $|v(x,t)-v(y,t)|\leq C_T p|x-y|^{1-2/p}$, for all $p>1$. Setting $p=|\ln |x-y||$ for $|x-y|<e$, we then get that $v$ is log-lipschitz locally uniformly in time: $|v(x,t)-v(y,t)|\leq C_T |x-y|(1+|\ln |x-y|)$ for $t\in [0,T)$ and $x,y\in \mathcal{U}$. By Cauchy-Lipschitz theorem, we infer that for a.e $x\in \mathcal{U}$ the curve $t\mapsto X(x,t)$ is the unique solution in $C^1([0,T) ; \mathcal{U})$ to the ODE
\[
\frac{\dd X(x,t)}{\dd t}=v(X(x,t),t), \quad X(x,0)=x.
\]
Let us note that the divergence free condition immediately implies the conservation of the $L^p$ norm: $\|\omega(\cdot,t)\|_{L^p}=\|\omega_0\|_{L^p}$, for $t\in [0,T]$, for $p\in [1,\infty]$.

In the setting of \eqref{hyp:omega0}, the fact that the log-lipschitz constant of the velocity field $v^\eps$ depends on $\|\omega^{\varepsilon}(\cdot,t)\|_{L^\infty}=\|\omega^{\varepsilon}_0\|_{L^\infty}$, diverging possibly as $\varepsilon^{-2}$, constitutes a major difficulty. Indeed, this does not allow to control the distance between the supports of the components $\omega_i^\eps$ and $\omega_j^\eps$ uniformly in $\eps$. Such a uniform control will be included in the forthcoming definition of $T_{\varepsilon}$, see \eqref{def:T_eps}, and one of the main consequences of the strong localization will be to state that $T_{\varepsilon}=T$.

\begin{rem}\label{rem:omega_i}
 In view of the relation in terms of the flow map, it is natural to define the decomposition of $\omega^{\varepsilon}$ in Theorem~\ref{theo:main} as the transport of the decomposition of the initial data:
 \[\omega^{\varepsilon}_i(\cdot,t):=X_\eps(\cdot,t)_\#\omega^\epsilon_{i,0}
 =\omega^\epsilon_{i,0}\circ X_{\varepsilon}(\cdot,t)^{-1}.\]
\end{rem}

\subsection{Some useful properties of the matrix {\it K}}

We recall that the definition of $K$ is given at \eqref{def:K}, that $h > 0$ is given and that for every $x \in \R^2$ we denote by $|X| := \sqrt{|x|^2 + h^2}$. We observe that we have the following decomposition:
\begin{equation*}
 \forall x \in \R^2, \quad K(x) = I_2- \frac{1}{h^2+|x|^2} N(x)
\text{ with }
 N(x) = \begin{pmatrix} x_1^2 & x_1x_2 \\ x_1x_2 & x_2^2 \end{pmatrix}.
\end{equation*}
The eigenvalues of $N(x)$ are 0 and $|x|^2$ associated respectively to the eigenvectors $x^\perp$ and $x$. Moreover,
$$N(x)^2=|x|^2 N(x).$$
Consequently, we have the following lemma.
\begin{lemme}\label{lem:K_def_pos}
For every $x \in \R^2$, $K(x)$ is a symmetric positive-definite matrix, with eigenvalues 1 and $h^2/|X|^2$ so that
\begin{equation*}
 \det K(x) = \frac{h^2}{|X|^2}.
\end{equation*}
In particular, $K$ is uniformly elliptic in the bounded domain $\mathcal{U}$.
\end{lemme}
Setting 
\begin{equation*}
 \Lambda(x) = I_2 +\frac{1}{h|X|+h^2}N(x)=I_2 +\frac{1}{h|X|+h^2}
 \begin{pmatrix}x_1^2 & x_1 x_2 \\
 x_1x_2 & x_2^2\end{pmatrix},
\end{equation*}
we can check that its inverse is given by
\begin{equation*}
 \Lambda(x)^{-1} =I_2-\frac{1}{h|X|+|X|^2} N(x).
\end{equation*}
and that
\[
(\Lambda(x)^{-1})^2=K(x), \quad \forall x \in \R^2.
\]

In order to study the elliptic problem $\mathcal{L}\Psi = \omega$ on $\mathcal{U}$, it will be useful to introduce a $C^1$-diffeomorphism $T$ such that $\Diff \mathcal{T}$ is proportional to $K^{-1/2}$.

\begin{lemme}\label{lem:existence_difféo}
There exists a radial function $f \in C^\infty\big(\R^2,[1,+\infty)\big)$ and a $C^1$-diffeomorphism $\mathcal{T} : \R^2 \to \R^2$ such that $\Diff \mathcal{T}(x) = f(x) K(x)^{-1/2} = f(x) \Lambda(x)$.
\end{lemme}
\begin{proof}
Let $f(x) = \rho(x_1^2 + x_2^2)$ where $\rho \in C^1(\R_+,\R_+^*)$ will be chosen later. Now let 
\begin{equation*}
 \forall x \in \R^2, \quad \mathcal{T}(x) := f(x) x = \rho(x_1^2+x_2^2)\begin{pmatrix}x_1 \\ x_2\end{pmatrix}.
\end{equation*}
We compute:
\begin{align*}
 \Diff \mathcal{T}(x) &= \rho(x_1^2+x_2^2) I_2 + 2\rho'(x_1^2+x_2^2) \begin{pmatrix}x_1^2 & x_1 x_2 \\ x_1 x_2 & x_2^2\end{pmatrix} 
 = \rho(x_1^2+x_2^2) I_2 + 2\rho'(x_1^2+x_2^2)N(x) \\
 &= \rho(x_1^2+x_2^2) \Big( I_2 + \frac{2\rho'(x_1^2+x_2^2) }{\rho(x_1^2+x_2^2)} N(x) \Big).
\end{align*}
This means that $\Diff \mathcal{T}(x) = f(x) \Lambda(x)$ if and only if 
\begin{equation*}
 \frac{2\rho'(x_1^2+x_2^2) }{\rho(x_1^2+x_2^2)} = \frac{1}{h|X|+h^2},
\end{equation*}
which also writes
\begin{equation}\label{eq:contrainte_rho}
 \frac{\rho'(s)}{\rho(s)} = g(s),\quad \forall s \ge 0,
\end{equation}
with 
\begin{equation*}
g(s) = \frac{1}{2(h \sqrt{s + h^2} + h^2)}.
\end{equation*}
We then take
\begin{equation*}
 \rho(s) = \exp \left( \int_0^s g(u) \dd u\right) \ge 1,\quad \forall s \ge 0,
\end{equation*}
and check that \eqref{eq:contrainte_rho} holds true, hence that $\Diff \mathcal{T}(x) = f(x) \Lambda(x)$. 

Now we need to prove that $\mathcal{T}$ is a $C^1$-diffeomorphism. By construction, it is clear that it is a $C^1$ map. Now we assume that for some $x,y \in \R^2$ we have $\mathcal{T}(x)=\mathcal{T}(y)$. By the definition of $\mathcal{T}$ this implies first that $|x|\rho(|x|^2)=|y|\rho(|y|^2)$, with $s\mapsto s\rho(s^2)$ strictly increasing, thus $|x|=|y|$. Finally, we get $x=y$, therefore $\mathcal{T}$ is injective. Moreover, by construction, the matrix $\Diff \mathcal{T}(x) = f(x)\Lambda(x)$ is invertible for every $x \in \R^2$, so by the global inverse function theorem, $\mathcal{T}$ is a $C^1$-diffeomorphism from $\R^2$ to $\mathcal{T}(\R^2)$. Finally, $g(s) \sim \frac{1}{2h\sqrt{s}}$ as $s \to \infty$ so $\rho(s) \to +\infty$, and therefore, $\mathcal{T}(\R^2) = \R^2$.
\end{proof}

The function $\mathcal{T}$ constructed in the previous proof belongs to $C^2(\R^2)$. Using also $\mathcal{T}^{-1}\in C^1(\R^2)$, we infer from the mean value theorem that there exists $C > 0$ such that
\begin{equation}\label{eq:Diff_T_Lip}
 |\Diff \mathcal{T}(x) - \Diff \mathcal{T}(y)| \le C |x-y| \le C^2 |\mathcal{T}(x)-\mathcal{T}(y)|\le C^3 |x-y|, \quad \forall x,y \in B(0,R_{\cal U}).
\end{equation}

\subsection{Expansion of the Biot-Savart law}

This section is devoted to the construction of a suitable decomposition of the solution of the problem $\mathcal{L} \Psi = \omega$. Indeed, we expect that $v^\epsilon =\nabla^\perp \Psi^\epsilon$ diverges as $\mathcal{O}(\epsilon^{-1})$ (at least pointwise) when the vorticity $\omega^\epsilon$ is concentrated. Thus it will be crucial to obtain an explicit formula for the most singular term to find some cancellations by symmetry properties. This kind of expansion will also be the key to find the term of order $|\ln \epsilon|$, which will not give symmetry cancellation and will give rise to the motion of the vortex filament. Such an expansion is one of the main tool in the studies of the vortex rings (see \cite{Mar2, Mar3}) and of the lake vortices (see \cite{Hientzsch_Lacave_Miot_2022_Dynamics_of_PV_for_the_lake_eq}). For helical flows, an expansion was derived by Cao and Wan in \cite{Cao_Wan_2022_Structure_of_Green} such that the remainder term is bounded in $C^1$ which is not enough for our study. Therefore, we propose an independent proof where we find a different leading term, such that the remainder will be bounded in $C^2$ uniformly in $\eps$.

As we expect that such an explicit expansion can be useful for other problems, we establish the following proposition with a general matrix $K$ satisfying only its relevant properties: those established in our case in Lemmas~\ref{lem:K_def_pos} and \ref{lem:existence_difféo}.

\begin{prop}\label{prop:Biot_Savart}
Let $\mathcal{U} \subset \R^2$ be a bounded simply-connected domain with $C^{1,1}$ boundary. Let $K\in C^1(\R^2;\mathcal{M}_2(\R))$ such that
\begin{enumerate}[(i)]
 \item For every $x \in \overline{\mathcal{U}},$ $K(x)$ is a symmetric positive-definite matrix.

 \item There exists a $C^1$-diffeomorphism $\mathcal{T} : \R^2 \to \R^2$ and a function $f : \R^2 \to \R_+^*$ such that $\Diff \mathcal{T}(x) = f(x) K^{-1/2}(x)$.
\end{enumerate}
Then, for every $\omega \in L^\infty_c(\mathcal{U})$ the unique solution $\Psi \in H^1_0(\mathcal{U})$ of $\cal L \Psi =\omega$ belongs to $\bigcap_{1 \le p <\infty}W^{2,p}(\mathcal{U})$ and is given by the expression
\begin{equation}\label{eq:val_psi}
 \Psi(x) = \int_{\mathcal{U}} \mathcal{G}_{K,\mathcal{U}}(x,y)\omega(y) \dd y
\qquad \text{where}\quad
 \mathcal{G}_{K,\mathcal{U}} := G_K + S_{K,\mathcal{U}},
\end{equation}
with
\begin{equation}\label{def:G_K-0}
 \forall x,y \in \mathcal{U}, \; x \neq y, \quad G_K(x,y) = \frac{1}{2\pi} \Big( \det K(x) \det K(y) \Big)^{-1/4} \ln|\mathcal{T}(x)-\mathcal{T}(y)|,
\end{equation}
and for all $y\in \mathcal{U}$, $x \mapsto S_{K,\mathcal{U}}(x,y)=S_{K,\mathcal{U}}(y,x)$ belongs to $\bigcap_{1 \le p <\infty}W^{2,p}(\mathcal{U})$ and is the unique solution in $H^1(\mathcal{U})$ of 
\begin{equation}\label{eq:pb_S}
 \left\{\begin{aligned}
& \mathcal{L} S_{K,\mathcal{U}}(x,y) = 
 -\frac{(\det K(y))^{-1/4} }{2\pi}\ln|\mathcal{T}(x)-\mathcal{T}(y)| \nabla\cdot\left(K(x)\nabla(\det K(x))^{-1/4} \right)\quad \forall x \in \mathcal{U}, \\ 
& S_{K,\mathcal{U}}(x, y ) = - G_{K}(x,y) \quad 
 \forall x \in \partial\mathcal{U}.
\end{aligned}\right.
\end{equation}
\end{prop}

Before proving Proposition~\ref{prop:Biot_Savart}, let us observe that the boundary condition for $S_{K,\cal U}$ is imposed in order to comply with the boundary condition $\mathcal{G}_{K,\mathcal{U}}(x,y) = 0$ if $x \in \partial \cal U$, so that eventually $\Psi = 0$ on $\partial \mathcal{U}$. Moreover, by hypothesis $(ii)$, $G_K$ is also given by the expression
\begin{equation*}
 G_K(x,y) = \frac{1}{2\pi} \frac{\sqrt{\det \Diff \mathcal{T}(x)}\sqrt{ \det \Diff \mathcal{T}(y)}}{f(x)f(y)} \ln|\mathcal{T}(x)-\mathcal{T}(y)|.
\end{equation*}

In addition, we observe that if $K(x) = \frac{1}{b(x)} I_2$, with $b \geq c > 0$ then it satisfies the hypotheses of Proposition~\ref{prop:Biot_Savart} with $\cal T(x) = x$ and $f(x)=1/\sqrt{b(x)}$. The expression we obtain for the Green's kernel coincide with the one obtained for the lake equation as in \cite[Proposition 3.1]{dekeyser-vanSchaftingen_2020}. Taking $K = I_2$ gives the usual 2D Green's kernel.

\begin{proof}
Let $\omega \in L^\infty_c(\mathcal{U})$. By hypothesis $(i)$, the operator $\mathcal{L}$ given in \eqref{def:L} is uniformly elliptic with coefficients belonging to $C^0(\overline{\mathcal{U}})$. As $\omega \in \cap_{p\geq 1} L^p(\cal U)$, we know from \cite[ Theorem 9.15]{Gilbarg_Trudinger_2001_elliptic} that the unique variational solution belongs to $W^{2,p}(\mathcal{U})$ for every $p \in [1,+\infty)$.

Noticing that for every $V\Subset \cal U$ and $y \in V$ fixed, the right hand side term of the first equality in \eqref{eq:pb_S} belongs to $\cap_{p\geq 1} L^p(\cal U)$ whereas the boundary condition is $C^\infty(\partial\mathcal{U})$ (both uniformly with respect to $y\in V$), we also deduce by
\cite[Theorem 9.15]{Gilbarg_Trudinger_2001_elliptic}, that the
 unique solution $x\mapsto S_{K,\mathcal{U}}(x,y)$ to \eqref{eq:pb_S} belongs to $W^{2,p}(\cal U)$ for every $p \geq 1$, uniformly with respect to $y\in V$. In particular, the function $x\mapsto S_{K,\mathcal{U}}(x,y)$ is $C^1$ and bounded in $W^{1,\infty}(\cal U)$ uniformly with respect to $y\in V$. 

In order to justify many computations in the following, we consider first the case of a smooth source term which we denote by $w$. More precisely, we let $w\in C^\infty_{c}(\cal U)$ be fixed and let
\begin{equation*}
 \Psi_1(x) := \int_{\mathcal{U}} G_K(x,y) w(y) \dd y, \quad \Psi(x)=\Psi[w](x):= \int_{\mathcal{U}} (G_K+S_{K,\cal U})(x,y) w(y) \dd y.
\end{equation*}
We compute in the sense of distribution the quantity $\langle \mathcal{L} \Psi_1 \, , \, \varphi \rangle$. Let $\varphi \in C_{\text{c}}^\infty(\mathcal{U})$ be a test function, then the functions are regular enough to differentiate under the integral sign and to apply Fubini's theorem so that
\begin{align*}
 \langle \mathcal{L} \Psi_1 \, , \, \varphi \rangle & = -\int_{\mathcal{U}}\left[K(x) \nabla \Psi_{1}(x)\right]\cdot \nabla\varphi(x) \dd x =-\int_{\mathcal{U}}\int_{\mathcal{U}} \left[K(x) \nabla_x G_K(x,y)\right]\cdot \nabla\varphi(x) w(y) \dd x \dd y \\
 & = -\int_{\mathcal{U}} \frac{\sqrt{ \det \Diff \mathcal{T}(y)}}{2\pi f(y)} \int_{\mathcal{U}} \left[K(x) \nabla_x\left(\frac{\sqrt{\det \Diff \mathcal{T}(x)}}{f(x)} \ln|\mathcal{T}(x)-\mathcal{T}(y)|\right)\right] \cdot \nabla \varphi(x) \dd x \, w(y) \dd y \\
 & = -\int_{\mathcal{U}} \frac{\sqrt{ \det \Diff \mathcal{T}(y)}}{2\pi f(y)} \int_{\mathcal{U}} \left[K(x) \frac{\sqrt{\det \Diff \mathcal{T}(x)}}{f(x)} \Diff \mathcal{T}^T(x) \frac{\mathcal{T}(x)-\mathcal{T}(y)}{|\mathcal{T}(x)-\mathcal{T}(y)|^2} \right]\cdot \nabla \varphi(x) \dd x \, w(y) \dd y \\
 & \sp \sp -\int_{\mathcal{U}} \frac{\sqrt{ \det \Diff \mathcal{T}(y)}}{2\pi f(y)} \int_{\mathcal{U}} K(x) \nabla\left(\frac{\sqrt{\det \Diff \mathcal{T}(x)}}{f(x)}\right) \ln|\mathcal{T}(x)-\mathcal{T}(y)|\cdot \nabla \varphi(x) \dd x \, w(y) \dd y \\
 & := - \int_{\mathcal{U}} \frac{\sqrt{ \det \Diff \mathcal{T}(y)}}{2\pi f(y)}(A_1(y) +A_2(y)) w(y) \dd y.
\end{align*}
For the term $A_1$, let us notice first that $K$ and $\Diff \mathcal{T}$ are symmetric matrices so that
\begin{equation}\label{eq:nablacirc}
 \nabla \varphi(x) = \nabla (\varphi \circ \mathcal{T}^{-1} \circ \mathcal{T} )(x)= \Diff\mathcal{T}(x) \nabla(\varphi\circ \mathcal{T}^{-1})(\mathcal{T}(x))
\end{equation} and thus
\begin{align*}
 A_1(y) 
 & = \int_{\mathcal{U}} f(x)\sqrt{\det \Diff \mathcal{T}(x)} \frac{\mathcal{T}(x)-\mathcal{T}(y)}{|\mathcal{T}(x)-\mathcal{T}(y)|^2} \cdot \nabla(\varphi\circ \mathcal{T}^{-1})(\mathcal{T}(x)) \dd x \\
 &=\int_{\cal T(\mathcal{U})} f (\mathcal{T}^{-1}(z))\sqrt{\det \Diff \mathcal{T} (\mathcal{T}^{-1}(z)) } \frac{z - \mathcal{T}(y)}{|z-\mathcal{T}(y)|^2} \cdot \nabla (\varphi \circ \mathcal{T}^{-1})(z) |\det \Diff \mathcal{T}^{-1}(z)| \dd z
\end{align*}
where we have made the change of variable $z = \mathcal{T}(x)$.
Since 
\begin{equation}\label{eq:detDT-1}
 \det \Diff \mathcal{T}^{-1}(z) = \frac{1}{\det \Diff \mathcal{T} (\mathcal{T}^{-1}(z))}
\end{equation} 
we obtain that
\begin{align*}
 A_1(y) =& \int_{\cal T(\mathcal{U})} f (\mathcal{T}^{-1}(z)) \frac{z - \mathcal{T}(y)}{|z-\mathcal{T}(y)|^2} \cdot \nabla (\varphi \circ \mathcal{T}^{-1})(z) \sqrt{\det \Diff \mathcal{T}^{-1}(z)} \dd z \\
 = &\int_{\cal T(\mathcal{U})} \frac{z - \mathcal{T}(y)}{|z-\mathcal{T}(y)|^2}\cdot \nabla \left(\varphi \circ \mathcal{T}^{-1} \times \sqrt{\det \Diff \mathcal{T}^{-1} } \times f \circ \mathcal{T}^{-1} \right)(z) \dd z \\
 & - \int_{\cal T(\mathcal{U})} f (\mathcal{T}^{-1}(z)) \frac{z - \mathcal{T}(y)}{|z-\mathcal{T}(y)|^2} \varphi(\mathcal{T}^{-1}(z)) \cdot \nabla \left(\sqrt{\det \Diff \mathcal{T}^{-1}}\right)(z) \dd z \\
 & - \int_{\cal T(\mathcal{U})} \frac{z - \mathcal{T}(y)}{|z-\mathcal{T}(y)|^2} \varphi(\mathcal{T}^{-1}(z)) \sqrt{\det \Diff \mathcal{T}^{-1}(z)} \cdot \nabla \left(f \circ \mathcal{T}^{-1} \right)(z) \dd z \\
 :=& A_{11}(y) + A_{12}(y) + A_{13}(y).
\end{align*}
Identifying that $\frac{z - \mathcal{T}(y)}{|z-\mathcal{T}(y)|^2} = 2\pi \nabla G_{\R^2}(z-\mathcal{T}(y))$, where $G_{\R^2}(\xi)=\frac{1}{2\pi}\ln |\xi|$
denotes the fundamental solution of the Laplacian on $\R^2$, 
we may write
\[
A_{11}(y) = - 2\pi \big(\varphi \circ \mathcal{T}^{-1} \times \sqrt{\det \Diff \mathcal{T}^{-1} } \times f \circ \mathcal{T}^{-1} \big)( \mathcal{T}(y))
 = - 2\pi \varphi(y) f(y) \sqrt{\det \Diff \mathcal{T}^{-1}(\mathcal{T}(y))}.
\]
In conclusion, we have obtained that
\begin{equation}\label{eq:j}
 \langle \mathcal{L} \Psi_1 \, , \, \varphi \rangle = \int_{\mathcal{U}} \varphi(y)w(y)\dd y - \int_{\mathcal{U}} \frac{\sqrt{ \det \Diff \mathcal{T}(y)}}{2\pi f(y)}(A_{12}(y) + A_{13}(y) +A_2(y)) w(y) \dd y,
\end{equation}
where the second right hand side integral motivates our definition of $S_{K,\cal U}$. To recognize $S_{K,\cal U}$, we come back to the variable $x = \mathcal{T}^{-1}(z)$:
\begin{align*}
 A_{12}(y) + A_{13}(y) & = - \int_{\mathcal{U}} f (x) \frac{\mathcal{T}(x) - \mathcal{T}(y)}{|\mathcal{T}(x)-\mathcal{T}(y)|^2} \varphi(x) \cdot \nabla \left(\sqrt{\det \Diff \mathcal{T}^{-1}}\right)(\mathcal{T}(x))|\det \Diff \mathcal{T}(x)| \dd x \\
 & \sp \sp - \int_{\mathcal{U}} \frac{\mathcal{T}(x) - \mathcal{T}(y)}{|\mathcal{T}(x)-\mathcal{T}(y)|^2} \varphi(x) \sqrt{\det \Diff \mathcal{T}^{-1}(\mathcal{T}(x))} \cdot \nabla \left(f \circ \mathcal{T}^{-1} \right)(\mathcal{T}(x))|\det \Diff \mathcal{T}(x)| \dd x
\end{align*}
We now observe using again relations~\eqref{eq:nablacirc} and~\eqref{eq:detDT-1} that
\begin{align*}
 \nabla \left(\sqrt{\det \Diff \mathcal{T}^{-1}}\right)(\mathcal{T}(x)) & = (\Diff \mathcal{T}(x))^{-1} \nabla\left(\sqrt{\det \Diff \mathcal{T}^{-1} \circ \mathcal{T}}\right)(x) \\
 &= \frac{1}{(f(x))^2}K(x)\Diff \mathcal{T}(x) \nabla\left( \frac{1}{\sqrt{\det \Diff \mathcal{T}}}\right)(x) \\
 & = -\frac{1}{(f(x))^2}K(x)\Diff \mathcal{T}(x)\frac{\nabla\left(\sqrt{\det \Diff \mathcal{T}}\right)(x)}{\det \Diff \mathcal{T}(x)}
\end{align*}
and similarly,
\begin{align*}
 \nabla (f \circ \mathcal{T}^{-1})(\mathcal{T}(x)) = (\Diff \mathcal{T}(x))^{-1} \nabla (f \circ \mathcal{T}^{-1} \circ \mathcal{T})(x) = \frac{1}{(f(x))^2}K(x)\Diff \mathcal{T}(x) \nabla f(x).
\end{align*}
Therefore, using once again that $\Diff \mathcal{T}(x)$ is symmetric, that it commutes with $K(x)$, and recalling that $\det D \mathcal{T}>0$,
\begin{align*}
 A_{12}(y) + A_{13}(y) & = \int_{\mathcal{U}} \Diff \mathcal{T}(x)\frac{\mathcal{T}(x) - \mathcal{T}(y)}{|\mathcal{T}(x)-\mathcal{T}(y)|^2}\varphi(x) \cdot K(x)\frac{\nabla\left(\sqrt{\det \Diff \mathcal{T}}\right)(x)}{f(x)} \dd x \\
 & \sp \sp + \int_{\mathcal{U}} \Diff \mathcal{T}(x)\frac{\mathcal{T}(x) - \mathcal{T}(y)}{|\mathcal{T}(x)-\mathcal{T}(y)|^2} \varphi(x) \cdot K(x)\frac{-\nabla f(x) }{(f(x))^2}\sqrt{\det \Diff \mathcal{T}(x)} \dd x \\
 & = \int_{\mathcal{U}} \Diff \mathcal{T}(x)\frac{\mathcal{T}(x) - \mathcal{T}(y)}{|\mathcal{T}(x)-\mathcal{T}(y)|^2} \varphi(x) \cdot K(x)\nabla \left( \frac{\sqrt{\det \Diff \mathcal{T}(x)}}{f(x)}\right) \dd x
\end{align*}
and thus recalling that
\begin{equation*}
 A_2(y) = \int_{\mathcal{U}} K(x) \nabla\left(\frac{\sqrt{\det \Diff \mathcal{T}(x)}}{f(x)}\right) \ln|\mathcal{T}(x)-\mathcal{T}(y)| \cdot \nabla \varphi(x) \dd x \\
\end{equation*}
we obtain, by an integration by parts, that
\begin{align*}
 A_{12}(y) + A_{13}(y) + A_2(y) 
 &= \int_{\mathcal{U}} K(x) \nabla\left(\frac{\sqrt{\det \Diff \mathcal{T}(x)}}{f(x)}\right) \cdot \nabla \Big( \ln|\mathcal{T}(x)-\mathcal{T}(y)| \varphi(x)\Big) \dd x\\
 &=-\int_{\mathcal{U}} \ln|\mathcal{T}(x)-\mathcal{T}(y)| \varphi(x) \nabla\cdot\Big( K(x) \nabla\left(\frac{\sqrt{\det \Diff \mathcal{T}(x)}}{f(x)}\right) \Big) \dd x.
\end{align*}
By the definition of $\cal L S_{K,\cal U}$ in \eqref{eq:pb_S} this allows us to conclude that
\[
-\int_{\mathcal{U}} \frac{\sqrt{ \det \Diff \mathcal{T}(y)}}{2\pi f(y)}(A_{12}(y) + A_{13}(y) +A_2(y)) w(y) \dd y = -\left\langle \cal L \int_{\cal U} S_{K,\cal U}(\cdot,y) w(y) \dd y ,\varphi \right\rangle.
\]
Finally, recalling \eqref{eq:j}, we have proved that
\begin{equation}\label{eq-Lf}
\langle \cal L \Psi[w],\varphi \rangle = \int_{\cal U} w(x) \varphi(x)\dd x
\end{equation}
for all $w$ and $\varphi$ belonging to $C^\infty_c(\cal U)$. Actually, the regularity of $\cal G_{K,\cal U}$ discussed at the beginning of the proof allows us to state that $\Psi[w]\in H^1_0(\cal U)$. Then, passing to the limit in the test functions, we infer that \eqref{eq-Lf} holds also true for $w\in C^\infty_c(\cal U)$ and $\varphi \in H^1_{0}(\cal U)$.

Thanks to this property, we are in position to establish the symmetry for $S_{K,\cal U}$. Indeed, for any $\varphi, w\in C^\infty_{c}(\cal U)$, we use \eqref{eq-Lf} with the test function $\Psi[\varphi]\in H^1_0(\cal U)$ and the symmetry of the matrix $K(x)$ for all $x\in \mathcal{U}$ to infer that
\[
\int w\Psi[\varphi] = \langle \cal L \Psi[w],\Psi[\varphi]\rangle=-\int \Big(K(x)\nabla\Psi[w]\Big)\cdot \nabla\Psi[\varphi] = \int \varphi \Psi[w].
\]
Thus, by the symmetry of $G_K$ (i.e. $G_K(x,y)=G_K(y,x)$)
we obtain
\[
\int_{\cal U}\int_{\cal U} S_{K,\cal U}(x,y)w (x) \varphi(y)\dd y\dd x=
\int_{\cal U}\int_{\cal U} S_{K,\cal U}(x,y)\varphi(x)w (y) \dd y\dd x
=\int_{\cal U}\int_{\cal U} S_{K,\cal U}(y,x)w (x) \varphi(y)\dd y\dd x
\]
and finally $S_{K,\cal U}(x,y)=S_{K,\cal U}(y,x)$.

Finally, we consider $\omega\in L^\infty_{c} (\cal U)$ and $(w_{n})$ a sequence of functions belonging to $C^\infty_{c}(\cal U)$ which converges to $\omega$ in $L^p(\mathcal{U})$ for some $p>2$. Without any loss of generality, we can assume that $\omega$ and $w_{n}$ for all $n$ are compactly supported in some $V\Subset \cal U$. Due to the $C^1$ regularity of $x\mapsto S_{K,\cal U}(x,y)$ uniformly in $y\in V$, we deduce that $\nabla\Psi[\omega]=\int \nabla_{x} \cal G_{K,\cal U} \omega(y)\dd y$ and that $\Psi[w_{n}]\to \Psi[\omega]$ when $n\to \infty$ in $W^{1,\infty}(\cal U)$. In particular, we may pass at the limit in \eqref{eq-Lf} to conclude that $\Psi[\omega]\in H^1_{0}(\cal U)$ is solution of $\cal L \Psi =\omega$.
\end{proof}

\subsection{Decomposition of the velocity}

We now come back to our problem with $K$ being the matrix given by \eqref{def:K}, which satisfies the hypotheses of Proposition~\ref{prop:Biot_Savart}. Let us give some details on the singular part $G_K$ of the Biot-Savart kernel.

We now denote by $H$ the kernel
\begin{equation}\label{def:H}
 H(x,y) := \frac{1}{2\pi}\Big( \det K(x) \det K(y) \Big)^{-1/4} = \frac{\sqrt{|X||Y|}}{2 \pi h},
\end{equation}
where the last equality comes from Lemma~\ref{lem:K_def_pos}, so that
\begin{equation}\label{def:G_K}
 G_K(x,y) = H(x,y)\ln|\mathcal{T}(x)-\mathcal{T}(y)|.
\end{equation}
We now compute:
\begin{equation}\label{eq:decomp_G_K}
 \nabla_x^\perp G_K(x,y) = \nabla_x^\perp H(x,y) \ln|\mathcal{T}(x)-\mathcal{T}(y)| + H(x,y) \left(\Diff \mathcal{T}(x)\frac{\mathcal{T}(x)-\mathcal{T}(y)}{\big|\mathcal{T}(x)-\mathcal{T}(y) \big|^2}\right)^\perp.
\end{equation}

Moreover,
\begin{equation*}
 \nabla_x^\perp H(x,y) = \frac{1}{4\pi h}\frac{x^\perp}{|X|} \frac{\sqrt{|Y|}}{\sqrt{|X|}} = \frac{H(x,y)}{2} \frac{x^\perp}{|X|^2}.
\end{equation*}

A consequence of Lemma~\ref{lem:existence_difféo}, Proposition~\ref{prop:Biot_Savart} (in particular \eqref{eq:val_psi}), and of the latter computations is that the velocity field $v^\eps$ in System~\eqref{eq:Helico2D_rescaled} decomposes as follows.
\begin{prop}\label{prop:decomp_vitesse} Let $(v^\eps, \omega^\eps)$ be a weak bounded solution of \eqref{eq:Helico2D_rescaled} in the sense of Definition~\ref{defi:weaksol2D}. Then,
\begin{equation}\label{eq:decomp:v}
 v^\eps = v_K^\eps + v_L^\eps + v_R^\eps,
\end{equation}
with $v_K^\eps$ defined as
\begin{equation}\label{def:v_K}
\begin{aligned}
 v_K^\eps(x,t) &= \Bigg(\int_{\mathcal{U}} H(x,y) \Diff \mathcal{T}(x) \frac{\mathcal{T}(x)-\mathcal{T}(y)}{\big|\mathcal{T}(x)-\mathcal{T}(y) \big|^2}\omega^\eps(y,t)\dd y\Bigg)^\perp,
 \end{aligned}
\end{equation}
$v_L^\eps$ defined as
\begin{equation}\label{def:v_L}
 v_L^\eps(x,t) = \frac{x^\perp}{2|X|^2}\int_{\mathcal{U}} G_K(x,y) \omega^\eps(y,t)\dd y,
\end{equation}
and
\begin{equation}\label{def:v_R}
 v_R^\eps(x,t) = \int_{\mathcal{U}} \nabla_x^\perp S_{K,\mathcal{U}}(x,y) \omega^\eps(y,t)\dd y.
\end{equation}
\end{prop}

When $\omega^\eps$ is close to a Dirac mass, as in Theorem~\ref{theo:main}, the part $v_K^\eps$ of the velocity is the most singular, of order $1/\epsilon$, however as usual in the study of 2D point-vortices, has a symmetric structure and will give the standard spinning around the filament which will not contribute to the displacement of the vortex core. The part $v_L^\eps$, also singular as it is of order $|\ln \eps|$ near the singularity, induces a rotation (and a vertical translation) of the helix around the origin at speed of order 1 thanks to the rescaling in System~\eqref{eq:Helico2D_rescaled}. Finally, the part $v_R^\eps$ of the velocity is a bounded remainder, whose contribution to the movement goes to $0$ as $\eps \to 0$ under the rescaling.

We conclude with an estimate on $\nabla G_K$.
\begin{lemme}\label{lem:maj_nabla_GK}
For every $R > 0$, there exists a constant $C_R$ such that for every $x,y \in B(0,R)$, $x \neq y$, there holds that
\begin{equation*}
 |\nabla_x G_K(x,y)| \le \frac{C_R}{|x-y|}.
\end{equation*}
\end{lemme}
\begin{proof}
We use the following decomposition :
\begin{align*}
 \nabla_x G_K(x,y) & = \frac{H(x,y)}{2} \frac{x}{|X|^2} \ln |\mathcal{T}(x) - \mathcal{T}(y)| + H(x,y) \Diff \mathcal{T}(x) \frac{\mathcal{T}(x)-\mathcal{T}(y)}{\big|\mathcal{T}(x)-\mathcal{T}(y) \big|^2} \\
 & := A_1(x,y) + A_2(x,y).
\end{align*}
Let $R > 0$. Let us notice that for every $x \in B(0,R)$, 
\begin{equation}\label{eq:minorationX}
 h\le |X| \le \sqrt{h^2 +R^2}
\end{equation}
hence $H$ is a bounded map on $B(0,R)$. So are $\Diff \mathcal{T}$ and $\mathcal{T}^{-1}$, so \eqref{eq:Diff_T_Lip} gives a constant $C$ such that
\begin{equation*}
 |A_2(x,y)| \le \frac{C}{|x-y|}.
\end{equation*}
Using again \eqref{eq:Diff_T_Lip}, we have a constant $C$ such that
\begin{equation*}
\Big|\ln|\mathcal{T}(x)-\mathcal{T}(y)| \Big|\le \Big|\ln |x-y| \Big| + C
\end{equation*}
so we have for every $x,y \in B(0,R)$ that
\begin{equation*}
 |x-y| \Big|\ln|\mathcal{T}(x)-\mathcal{T}(y)| \Big| \le C.
\end{equation*}
By \eqref{eq:minorationX} and since $H$ is bounded, we conclude that
\begin{equation*}
 |A_1(x,y)| \le \frac{C}{|x-y|}.
\end{equation*}
The constants involved only depend on $R$ so our proof is complete.
\end{proof}

\section{Reduction of the problem to a single vortex}\label{sec:reduction}

In order to prove Theorem~\ref{theo:main}, we introduce a reduced problem that focuses on the dynamics of a single blob of vorticity, by considering the effect of the other blobs as an exterior field.

To proceed, we place ourselves in the framework of Theorem~\ref{theo:main}, namely we consider $\omega^\eps(\cdot,t) = \sum_{i=1}^N \omega_i^\eps(\cdot,t)$ the solution of \eqref{eq:Helico2D_rescaled} with initial data $\omega^\eps(\cdot,0) = \sum_{i=1}^N \omega_{i,0}^\eps$ satisfying \eqref{hyp:omega0}. As in Remark~\ref{rem:omega_i}, $\omega_i^\eps$ corresponds to the transport of $\omega_{i,0}$.

For each $i \in \{1,\ldots,N\}$, we introduce the exterior field $F^\eps_i : \cal U\times[0,+\infty) \to \R^2$ given by
\begin{equation}\label{def:F}
 F_i^\eps(x,t) = \sum_{j \neq i} \int_\mathcal{U} \nabla_x^\perp\mathcal{G}_{K,\mathcal{U}}(x,y)\omega_j^\eps(y,t)\dd y,
\end{equation}
so that the $i$-th blob $\omega_i^\eps$ satisfies the following equations:
\begin{equation}\label{eq:Helico2D_rescaled_F}
\begin{cases}
 \ds \partial_t \, \omega^\eps_i + \frac{1}{|\ln\eps|}(v^\eps_i + F^\eps_i)\cdot \nabla \omega^\eps_i = 0 & \text{ in } \cal U \times \R_+^* \vspace{1mm}\\
 v^\eps_i = \nabla^\perp \Psi^\eps_i & \text{ in } \cal U \times \R_+ \vspace{1mm} \\
 \div \big( K(x) \nabla \Psi^\eps_i \big) = \omega^\eps_i & \text{ in } \cal U \times \R_+ \vspace{1mm} \\
 \omega^\eps_i(\cdot, 0) = \omega^\eps_{i,0} & \text{ in } \cal U.
\end{cases}
\end{equation}
For any $r \ge 0$ and $\eta >0$ we define the annulus $\mathcal{A}_\eta^r$ as
\begin{equation}\label{def:Aeta}
 \mathcal{A}_{\eta}^r = \left\{ x \in \R^2, \big||x|-r\big| < \eta \right\}.
\end{equation}
Now recall from the hypotheses of Theorem~\ref{theo:main} that $\mathcal{U}$ contains a ball $B(0,R_\cal U)$, and that for ever $i \in \{1,\ldots,N\}$, $|z_{i,0}| < R_\cal U$. Let
\begin{equation*}
 \eta_0 = \frac{1}{4}\min \Bigg(\left\{ \Big||z_{i,0}| - |z_{j,0}| \Big| \, , \, i \neq j \right\} \bigcup \Big\{ R_\mathcal{U} - |z_{i,0}| \, , \, i \in \{1,\ldots,N\}\Big\}\Bigg),
\end{equation*}
and let $T > 0$. From the definition of $\eta_0$, we infer in particular that $\mathcal{A}_{\eta_0}^{|z_{i,0}|} \subset \mathcal{U}$ and more precisely that
\begin{equation}\label{eq:there_is_dist}
 \dist\left( \bigcup_{i=1}^N\mathcal{A}_{\eta_0}^{|z_{i,0}|} \, , \, \partial\mathcal{U}\right) \ge \eta_0. 
\end{equation}
This gives the following proposition.
\begin{prop}\label{prop:S_K_borné_Lip}
There exists a constant $C$ such that for every $x,y \in \bigcup_{j=1}^N\mathcal{A}_{\eta}^{|z_{j,0}|}\times \mathcal{U}$,
\begin{align*}
 |S_{K,\mathcal{U}}(x,y) | & \le C \\
 |\nabla_x S_{K,\mathcal{U}}(x,y) | & \le C.
\end{align*}
\end{prop}
\begin{proof}
This is a direct consequence of relation~\eqref{eq:there_is_dist} and the fact exposed in the proof of Proposition~\ref{prop:Biot_Savart} that for every $V \Subset \mathcal{U}$, $x\mapsto S_{K,\mathcal{U}}(x,y)$ is bounded in $W^{1,\infty}(\mathcal{U})$ uniformly in $y$ on $V$. 
\end{proof}

We introduce
\begin{equation}\label{def:T_eps}
 T_\eps = \sup \left\{ t \in [0,T] \, , \, \forall s \in [0,t] \, , \, \forall i \in \{1,\ldots,N\} \, , \quad \supp \omega_i^\eps(\cdot,s) \subset \mathcal{A}_{\eta_0}^{|z_{i,0}|}\right\},
\end{equation}
which is a time during which every blob of vorticity is localized in a fixed annulus. By continuity of the trajectories, we know that $T_\eps > 0$ for every $\eps > 0$. 

On time interval $[0,T_\eps]$, we have the following useful estimates on $F_i^\eps$.
\begin{lemme}\label{lem:hyp_F}
There exists a constant $C$ such that for every $\eps > 0$ small enough, for every $i \in \{1,\ldots,N\}$, for every $x,y \in\mathcal{A}_{\eta_0}^{|z_{i,0}|}$ and for every $t \in [0,T_\eps]$,
\begin{equation*}
 \left\{\begin{aligned}
& \big| F_i^\eps(x,t) - F_i^\eps(y,t) \big| \le C|x-y| \\
& | F_i^\eps(x,t) | \le C, \\
& \div F_i^\eps(x,t) = 0.
\end{aligned}\right.
\end{equation*}
\end{lemme}
\begin{proof}
Let $i \in \{1,\ldots,N\}$ and $t \in [0,T_\eps]$. By definition of $F_i^\eps$ (given in \eqref{def:F}), $\div F_i^\eps = 0$, for every $x \in \cal U$.

Now let $x,y \in \mathcal{A}_{\eta_0}^{|z_i,0|}$. Then
\begin{equation*}
 F_i^\eps(x,t) - F_i^\eps(y,t) = \sum_{j \neq i} \int_\mathcal{U} \big(\nabla_x^\perp \mathcal{G}_{K,\mathcal{U}}(x,z)-\nabla_x^\perp \mathcal{G}_{K,\mathcal{U}}(y,z)\big)\omega_j^\eps(z,t)\dd z.
\end{equation*}

From the definition of $G_K$ \eqref{def:G_K}, we note that $G_K(\cdot ,z) \in C^2 \big( \mathcal{A}_{\eta_0}^{|z_{i,0}|}\big)$ uniformly to $z\in \bigcup_{j \neq i} \mathcal{A}_{\eta_0}^{|z_{j,0}|}$. In the same way, the right hand side term of \eqref{eq:pb_S} has the same regularity, hence by elliptic regularity, $S_{K,\cal U}$ belongs also to $C^2 \big( \mathcal{A}_{\eta_0}^{|z_{i,0}|}\big)$ uniformly to $z\in \bigcup_{j \neq i} \mathcal{A}_{\eta_0}^{|z_{j,0}|}$. 
Therefore, drawing a curve included in $\mathcal{A}_{\eta_0}^{|z_{i,0}|}$ between any $x,y \in \mathcal{A}_{\eta_0}^{|z_{i,0}|}$, we conclude that there exists a constant $C$ such that for every $z \in \bigcup_{j \neq i} \mathcal{A}_{\eta_0}^{|z_{j,0}|}$ 
\begin{equation*}
 \big|\nabla_x^\perp \mathcal{G}_{K,\cal U}(x,z)-\nabla_x^\perp \mathcal{G}_{K,\cal U}(y,z)\big| \le C|x-y|,
\end{equation*}
and thus we have that
\[
 \big|F_i^\eps(x,t) - F_i^\eps(y,t)\big| \le C|x-y| \|\omega^\eps\|_{L^1}.
\]
Recalling that $ \|\omega^\eps\|_{L^1}$ does not depend on $\eps$, we have that
\[
 \big|F_i^\eps(x,t) - F_i^\eps(y,t)\big| \le C|x-y|.
\]
Using only the $C^1$ regularity of $\mathcal{G}_{K,\cal U}$ uniformly in $z$, we have in the same way
\[
 | F_i^\eps(x,t) | \le C.
\]
\end{proof}

\section{Single vortex in an exterior field}\label{sec:single_vortex}

In this section we turn to the study of the reduced problem with a single blob of vorticity. Let $i \in \{1,\ldots,N\}$. We denote by $z_0 = z_{i,0}$, $\gamma = \gamma_i$, $r_0 = |z_0|$,
\begin{equation*}
 \nu = -\frac{\gamma}{4\pi h\sqrt{h^2+r_0^2}},
\end{equation*}
and define
\begin{equation*}
 z(t) = \tilde{R}_{t\nu} z_0.
\end{equation*}
This way, we drop completely the index $i$ and simply consider a solution $\omega^\eps$ of \eqref{eq:Helico2D_rescaled_F} with an exterior field $F^\eps$. Without loss of generality, one can assume that $\gamma > 0$ so that the hypotheses on $\omega_0^\eps$ now become, for every $\eps >0$,
\begin{equation}\label{hyp:omega0_reduced}
 \left\{\begin{aligned}
 &\supp \omega^\eps_0 \subset B(z_0,\eps) \\
 & 0\leq \omega^\eps_0 \le M_0\eps^{-2} \\
& \int \omega^\eps_0(x)\dd x = \gamma . \\
 \end{aligned}\right.
\end{equation}
For the sake of clarity, we also denote by $\mathcal{A}_\eta$ the annulus $\mathcal{A}_\eta^{r_0}$ in this section since all of our study will take place near $r_0$.
By construction, $\omega^\eps$ satisfies
\begin{equation}\label{eq:a_priori_strong_loc}
 \forall t \in [0,T_\eps], \quad \supp \omega^\eps(\cdot,t) \subset \mathcal{A}_{\eta_0}
\end{equation}
and the exterior field $F^\eps$ satisfies by Lemma~\ref{lem:hyp_F} that for every $x,y \in\mathcal{A}_{\eta_0}$ and for every $t \in [0,T_\eps]$,
\begin{equation}\label{hyp:F_reduced}
 \left\{\begin{aligned}
& \big| F^\eps(x,t) - F^\eps(y,t) \big| \le C|x-y| \\
& | F^\eps(x,t) | \le C, \\
& \nabla \cdot F^\eps(x,t) = 0.
\end{aligned}\right.
\end{equation}
\begin{rem}
We recall that $T_\eps$ given by \eqref{def:T_eps} takes into account all the index $i$. Therefore, every estimate obtained in this section holds on the time interval $[0,T_\eps]$ for every blob simultaneously.
\end{rem}
This section is devoted to the proof of the following intermediate result describing how $\supp \omega^\eps(\cdot,t)$ ``mostly'' shrinks to $z(t)$ as $\eps \to 0$ at least on the time interval $T_\eps$.
\begin{theo}\label{theo:reduit}
The following properties hold true.
\begin{itemize}
 \item[$(i)$] There exists constants $C_T$ and $\eps_{T}$ such that for every $\eps \in (0,\eps_T]$, by letting $r_\eps = \left(\frac{\ln |\ln\eps|}{|\ln \eps|}\right)^{1/2}$, we have 
 \begin{equation*}
 \sup_{t \in [0,T_\eps]}\left| \gamma - \int_{B(z(t),r_\eps)} \omega^\eps(x,t)\dd x\right| \le \frac{C_T}{\ln |\ln \eps|},
 \end{equation*}
 and
 \begin{equation*}
 \sup_{t \in [0,T_\eps]} \left|\frac{1}{\gamma}\int x\omega^\eps(x,t)\dd x - z(t)\right| \le \frac{C_T}{\sqrt{|\ln\eps|}}.
 \end{equation*}
 
 \item[$(ii)$] For every $\kappa \in (0, \frac{1}{4})$, there exists constants $C_{\kappa,T}$ and $\eps_{\kappa,T}>0$, such that for every $\eps \in (0,\eps_{\kappa,T}]$ and for every $t \in [0,T_\eps]$, we have
 \begin{equation*}
 \supp \omega^\eps(\cdot,t) \subset \left\{ x \in \R^2 \, , \, \Big||x| - |z|\Big| \le \frac{C_{\kappa,T}}{|\ln\eps|^\kappa} \right\}.
 \end{equation*}
\end{itemize}
\end{theo}

The plan of the proof of Theorem~\ref{theo:reduit} is the following. We derive precise estimates on energy and vorticity moments of $\omega^\eps$ to obtain the weak localization $(i)$. Then we reproduce the now classical arguments (see \cite{Mar3,Hientzsch_Lacave_Miot_2022_Dynamics_of_PV_for_the_lake_eq}) to obtain the strong localization $(ii)$. Even though these properties are only obtained on the time interval $[0,T_\eps]$, the fact that the localization from $(ii)$ is better that the \emph{a priori} localization will yield in the end that $T_\eps = T$ for every $\eps$ small enough.

\subsection{Preliminary computations}
Let us introduce a few preliminary technical lemmas. We start with a useful formula, which is a consequence of Remark~\ref{rem:omega_i} (recalling that the flow $X_\eps$ associated to the velocity field $v_\eps + F_\eps$ is divergence free):
\begin{lemme}\label{lem:deriv}
Let $\alpha \in C^1(\cal U\times [0,T_\eps],\R)$. Then for every $t \in [0,T_\eps]$,
\begin{equation*}
 \der{}{t} \int \alpha(x,t) \omega^\eps(x,t) \dd x = \int \partial_t \alpha(x,t)\omega^\eps(x,t)+\frac{1}{|\ln\eps|}\int \nabla \alpha(x,t) \cdot \big(v^\eps(x,t)+F^\eps(x,t)\big)\omega^\eps(x,t) \dd x.
\end{equation*}
\end{lemme}

\begin{lemme}\label{lem:lnT}
There exists a constant $C$ independent of $\eps$ such that for $\eps$ small enough and for every $t \in [0,T_\eps]$, we have
\begin{equation*}
 \left|\iint \ln |\mathcal{T}(x)-\mathcal{T}(y)| \omega^\eps(x,t)\omega^\eps(y,t) \dd x \dd y\right| \le C |\ln \eps|. 
\end{equation*}
Moreover, at time $0$, we have 
\begin{equation*}
 \iint \ln |\mathcal{T}(x)-\mathcal{T}(y)| \omega^\eps_0(x)\omega^\eps_0(y) \dd x \dd y = -\gamma^2 |\ln\eps| + \mathcal{O}(1)
\end{equation*}
as $\eps \to 0$.
\end{lemme}
\begin{proof}
Notice first that
\begin{multline}\label{eq:totale}
 \iint \ln |\mathcal{T}(x)-\mathcal{T}(y)| \omega^\eps(x,t)\omega^\eps(y,t) \dd x \dd y = \iint \ln\frac{|\mathcal{T}(x)-\mathcal{T}(y)|}{|x-y|}\omega^\eps(x,t)\omega^\eps(y,t)\dd x \dd y \\ + \iint \ln|x-y|\omega^\eps(x,t)\omega^\eps(y,t)\dd x \dd y .
\end{multline}
By \eqref{eq:Diff_T_Lip}, there exists a constant $C$ such that for every $x,y \in \mathcal{A}_{\eta_0}$,
\begin{equation*}
 \left| \ln\frac{|\mathcal{T}(x)-\mathcal{T}(y)|}{|x-y|} \right| \le C,
\end{equation*}
hence, recalling that $\|\omega^\eps(\cdot,t)\|_{L^1}=\gamma$ is independent of $\eps$, 
\begin{equation}\label{eq:terme_important}
 \iint \ln\frac{|\mathcal{T}(x)-\mathcal{T}(y)|}{|x-y|}\omega^\eps(x,t)\omega^\eps(y,t)\dd x \dd y = \mathcal{O}(1)
\end{equation}
as $\eps \to 0$, namely is bounded uniformly in $\eps$.

We now observe that there exists a constant $C$ such that for every $x,y \in \mathcal{A}_{\eta_0}$, 
\[
\ln|x-y| \le C,
\]
and thus, recalling that $\omega^\eps \geq 0$,
\begin{equation}\label{eq:maj_below}
 \iint \ln |x-y| \omega^\eps(x,t)\omega^\eps(y,t) \dd x \dd y \le C.
\end{equation}
We now apply Lemma B.1 from \cite{Hientzsch_Lacave_Miot_2022_Dynamics_of_PV_for_the_lake_eq}, which is recalled as Lemma~\ref{lem:rearrangement}, to $g(s) = -\ln(s)\mathds{1}_{(0,1)}$, $M = M_0\eps^{-2}$. Indeed by \eqref{hyp:omega0_reduced} and since the $L^1$ and $L^\infty$ norms of $\omega^\epsilon$ are conserved, then $\omega^\eps(\cdot,t) \in \mathcal{E}_{M,\gamma}$ for every $t \in [0,T_\eps]$. Therefore, by letting $r = \eps\sqrt{\frac{\gamma}{\pi M_0}}$, we obtain that
\begin{multline*}
 -\int \ln|x-y|\omega^\eps(y,t)\dd y \le -2\pi M_0\eps^{-2} \int_0^r s \ln(s)\dd s \\ = -2\pi M_0\eps^{-2} \left(\frac{r^2 \ln r}{2} + \mathcal{O}(\eps^2)\right) = - \gamma\ln \eps + \mathcal{O}(1)= \gamma|\ln \eps| + \mathcal{O}(1)
\end{multline*}
as $\eps \to 0$, which yields that
\begin{equation}\label{eq:maj_above}
 \iint \ln |x-y| \omega^\eps(x,t)\omega^\eps(y,t) \dd x \dd y \ge -\gamma^2 |\ln \eps| + \mathcal{O}(1).
\end{equation}
Gathering \eqref{eq:totale}, \eqref{eq:terme_important}, \eqref{eq:maj_below} and \eqref{eq:maj_above}, we conclude that for $\eps$ small enough,
\begin{equation*}
 \left|\iint \ln |\mathcal{T}(x)-\mathcal{T}(y)| \omega^\eps(x,t)\omega^\eps(y,t) \dd x \dd y\right| \le C |\ln \eps|. 
\end{equation*}
At time 0, we improve the upper-bound obtained in \eqref{eq:maj_below} by using the strong localization hypothesis \eqref{hyp:omega0_reduced}. Indeed we have that
\begin{equation*}
 \forall x,y \in \supp \omega_0^\eps, \quad |x-y| \le 2\eps,
\end{equation*}
therefore, we get that
\begin{equation*}
 \iint \ln|x-y|\omega^\eps_0(x)\omega^\eps_0(y)\dd x \dd y \le \iint \ln(2\eps)\omega^\eps_0(x)\omega^\eps_0(y)\dd x \dd y \le -\gamma^2|\ln\eps| + \mathcal{O}(1),
\end{equation*}
which, combined with \eqref{eq:totale}, \eqref{eq:terme_important} and \eqref{eq:maj_above} gives that
\begin{equation*}
 \iint \ln |\mathcal{T}(x)-\mathcal{T}(y)| \omega^\eps_0(x,t)\omega^\eps_0(y,t) \dd x \dd y = -\gamma^2 |\ln\eps| + \mathcal{O}(1).
\end{equation*}
\end{proof}

\begin{rem}\label{rem:lnT}
We also have that 
\begin{equation*}
 \iint \Big|\ln |\mathcal{T}(x)-\mathcal{T}(y)|\Big| \omega^\eps(x,t)\omega^\eps(y,t) \dd x \dd y \le C |\ln \eps|
\end{equation*}
and
\begin{equation*}
 \iint \Big|\ln |\mathcal{T}(x)-\mathcal{T}(y)|\Big| \omega^\eps_0(x)\omega^\eps_0(y) \dd x \dd y = \gamma^2 |\ln\eps| + \mathcal{O}(1)
\end{equation*}
with very few adaptations to the proof of Lemma~\ref{lem:lnT} since the for every $t \in [0,T_\eps]$, $\omega^\eps(\cdot,t)$ is supported in $\mathcal{A}_{\eta_0}$ which is bounded.
\end{rem}

\subsection{Estimates on the local energy}

We introduce the local energy
\begin{equation}\label{def:psi}
 \psi^\eps(x,t) := \int G_K(x,y) \omega^\eps(y,t)\dd y.
\end{equation}
In particular, we have from \eqref{def:v_L} that
\begin{equation}\label{eq:v_L-psi}
 v_L^\eps(x,t) = \frac{x^\perp}{2|X|^2}\psi^\eps(x,t).
\end{equation}
We establish an important lemma on the local energy $\psi^\eps$ defined at \eqref{def:psi}. 
\begin{lemme}\label{lem:encadrement_psi}
There exists a constant $C$ such that for every $t\le T_\eps$ and $x \in \mathcal{A}_{\eta_0}$,
\begin{equation*}
 -C \le -\psi^\eps(x,t) \le \gamma\frac{|X|}{2\pi h}|\ln\eps| + \mathcal{O}(1),
\end{equation*}
as $\eps \to 0$.
\end{lemme}
\begin{proof}
We recall that
\begin{equation*}
 \psi^\eps(x,t) = \int G_K(x,y) \omega^\eps(y,t)\dd y = \int \frac{\sqrt{|X||Y|}}{2\pi h} \ln|\mathcal{T}(x)-\mathcal{T}(y)|\omega^\eps(y,t)\dd y.
\end{equation*}
Since for every $x \in \mathcal{A}_{\eta_0}$, $|x| \le r_0 + \eta_0$, and since $\mathcal{T}$ is bounded on $B(0,r_0+\eta_0)$, there exists a constant $C$ independent of $\eps$ such that
\begin{equation*}
 \psi^\eps(x,t) \le C.
\end{equation*}
Now we write that
\begin{equation*}
 \psi^\eps(x,t) = \frac{|X|}{2\pi h}\int \ln|\mathcal{T}(x)-\mathcal{T}(y)|\omega^\eps(y,t)\dd y + \frac{\sqrt{|X|}}{2\pi h}\int\big( \sqrt{|Y|} - \sqrt{|X|}\big) \ln |\mathcal{T}(x)-\mathcal{T}(y)| \omega^\eps(y,t)\dd y.
\end{equation*}
Since the map $x \mapsto \sqrt{|X|}$ is smooth on the set $\mathcal{A}_{\eta_0}$, there exists a constant $C$ such that for all $x,y \in \mathcal{A}_{\eta_0}$, $\left|\sqrt{|Y|} - \sqrt{|X|}\right| \le C|x-y|$ and therefore
\begin{equation*}
 \frac{\sqrt{|X|}}{2\pi h}\int\big( \sqrt{|Y|} - \sqrt{|X|}\big) \ln |\mathcal{T}(x)-\mathcal{T}(y)| \omega^\eps(y,t)\dd y = \mathcal{O}(1),
\end{equation*}
so that
\begin{align*}
 -\psi^\eps(x,t) & = -\frac{|X|}{2\pi h}\int \ln|\mathcal{T}(x)-\mathcal{T}(y)|\omega^\eps(y,t)\dd y + \mathcal{O}(1).
\end{align*}
Reproducing the arguments of Lemma~\ref{lem:lnT}, we obtain that
\begin{equation*}
 -\psi^\eps(x,t) \le \gamma\frac{|X|}{2\pi h}|\ln\eps| + \mathcal{O}(1).
\end{equation*}
\end{proof}

\subsection{Estimate on radial vorticity moments}
For every $k \ge 1$, for every $t \le T_\eps$, let
\begin{equation}\label{def:J}
 J_k^\eps(t) = \int |x|^k \omega^\eps(x,t)\dd x.
\end{equation}
Following the observation in \cite{Hientzsch_Lacave_Miot_2022_Dynamics_of_PV_for_the_lake_eq}, we obtain sharp estimates on $J_k^\eps$ using the fact that $v_L^\eps(x,t) \cdot x = 0$.
\begin{lemme}\label{lem:J} Assume that either $k\ge2$, or $k=1$ and $r_0\neq0$. Then for every $t \le T_\eps$ , there exists a constant $C_k$ such that
\begin{equation*}
\big|J_k^\eps(t) - \gamma r_0^k\big| \le \frac{C_k}{|\ln\eps|}.
\end{equation*}
\end{lemme}
\begin{proof}
Let $k\ge2$, or $k=1$ and $r_0\neq0$. We compute using Lemma~\ref{lem:deriv}:
\begin{align*}
 \der{}{t} J_k^\eps(t) & = \frac{k}{|\ln\eps|} \int |x|^{k-2} x \cdot (v^\eps + F^\eps)(x,t) \omega^\eps(x,t)\dd x.
\end{align*}
Then, using the decomposition~\eqref{eq:decomp:v} and \eqref{def:v_L}, we have that
\begin{align*}
 \der{}{t} J_k^\eps(t) & = \frac{k}{|\ln\eps|} \int |x|^{k-2} x \cdot v_K^\eps(x,t) \omega^\eps(x,t)\dd x + \frac{k}{|\ln\eps|} \int |x|^{k-2} x \cdot (v_R^\eps + F^\eps)(x,t) \omega^\eps(x,t)\dd x \\
 & := \frac{k}{|\ln\eps|} \big(A_1 + A_2\big).
\end{align*}
We start with $A_2$. Recalling that $F^\eps$ is bounded from \eqref{hyp:F_reduced}, and observing that $v_R^\eps$, defined at relation~\eqref{def:v_R}, is bounded by Proposition~\ref{prop:S_K_borné_Lip}, we obtain that 
\begin{equation*}
 A_2 = \mathcal{O}(1).
\end{equation*}
We then turn to $A_1$. We compute using the definition \eqref{def:v_K} of $v_K^\eps$ and the symmetry of $\Diff \cal T$:
\begin{align*} A_1 & =\iint H(x,y) |x|^{k-2} x \cdot \left(\Diff \mathcal{T}(x)\frac{\mathcal{T}(x)-\mathcal{T}(y)}{\big|\mathcal{T}(x)-\mathcal{T}(y) \big|^2}\right)^\perp\omega^\eps(y,t)\omega^\eps(x,t)\dd x \dd y \\
& = -\iint H(x,y) |x|^{k-2} x^\perp \cdot \Diff \mathcal{T}(x)\frac{\mathcal{T}(x)-\mathcal{T}(y)}{\big|\mathcal{T}(x)-\mathcal{T}(y) \big|^2}\omega^\eps(y,t)\omega^\eps(x,t)\dd x \dd y \\
 & = - \frac12\iint H(x,y) \Big(|x|^{k-2} \Diff \mathcal{T}(x) x^\perp - |y|^{k-2}\Diff \mathcal{T}(y) y^\perp \Big)\cdot\frac{\mathcal{T}(x)-\mathcal{T}(y)}{\big|\mathcal{T}(x)-\mathcal{T}(y) \big|^2}\omega^\eps(y,t)\omega^\eps(x,t)\dd x \dd y.
\end{align*}
Now we observe that since $k \ge 2$ or $r_0 \neq 0$, the map $x \mapsto |x|^{k-2} \Diff \mathcal{T}(x) x^\perp $ is smooth on $\mathcal{A}_{\eta_0}$ so for every $x,y \in \supp \omega^\eps(\cdot,t)$,
\begin{equation*}
 \left|\left(|x|^{k-2} \Diff \mathcal{T}(x) x^\perp - |y|^{k-2}\Diff \mathcal{T}(y)y \right)\cdot\frac{\mathcal{T}(x)-\mathcal{T}(y)}{\big|\mathcal{T}(x)-\mathcal{T}(y) \big|^2}\right| \le C\frac{|x-y|}{|T(x)-T(y)|} \le C,
\end{equation*}
and thus
\begin{equation*}
 A_1 = \mathcal{O}(1).
\end{equation*}

In conclusion, we proved that
\begin{equation*}
 \der{}{t} J_k^\eps(t) = \mathcal{O}\left(\frac{1}{|\ln\eps|} \right).
\end{equation*}
Since by \eqref{hyp:omega0_reduced},
\begin{equation*}
 |J_k^\eps(0) - \gamma r_0^k| = \left|\int \big(|x|^k-r_0^k\big)\omega_0^\eps(x)\dd x\right| = \mathcal{O}(\eps),
\end{equation*}
then,
\begin{equation*}
 J_k^\eps(t) - \gamma r_0^k = \mathcal{O}\left(\frac{1}{|\ln\eps|}\right).
\end{equation*}
\end{proof}

\subsection{Estimates on the energy}

Let
\begin{equation*}
 E^\eps(t) := - \iint \mathcal{G}_{K,\mathcal{U}}(x,y) \omega^\eps(x,t) \omega^\eps(y,t)\dd x \dd y.
\end{equation*}
We observe quickly that $E^\eps$ is the $2D$ energy of the helical solution, in the sense that
\begin{equation*}
 E^\eps(t) = \int \|K^{1/2}(x)v^\eps(x)^\perp\|^2 \dd x,
\end{equation*}
since
\begin{align*}
 \int \|K^{1/2}(x)v^\eps(x)^\perp\|^2 \dd x & = \int v^\eps(x,t)^\perp\cdot \big( K(x) v^\eps(x,t)^\perp\big) \dd x \\
 & = \int \nabla \Psi^\eps(x,t) \cdot \big(K(x)\nabla \Psi^\eps(x,t)\big) \dd x \\
 & = -\int \Psi^\eps(x,t)\nabla \cdot \big(K(x)\nabla \Psi^\eps(x,t)\big)\dd x \\
 & = -\int \Psi^\eps(x,t) \omega^\eps(x,t)\dd x \\
 & = -\iint \mathcal{G}_{K,\mathcal{U}}(x,y) \omega^\eps(x,t)\omega^\eps(y,t) \dd x \dd y \\
 & = E^\eps(t).
\end{align*}
We begin our study of $E^\eps$ with this first lemma.
\begin{lemme}\label{lem:E-psiomega}
For every $t \in [0,T_\eps]$, 
\begin{equation*}
 E^\eps(t) = -\iint G_K(x,y) \omega^\eps(x,t)\omega^\eps(y,t) \dd x \dd y + \mathcal{O}(1) = -\int\psi^\eps(x,t)\omega^\eps(x,t)\dd x + \mathcal{O}(1),
\end{equation*}
as $\eps \to 0$.
\end{lemme}
\begin{proof}
We observe first that
\begin{equation*}
 E^\eps(t) = -\iint G_K(x,y) \omega^\eps(x,t)\omega^\eps(y,t) \dd x \dd y - \iint S_{K,\mathcal{U}}(x,y) \omega^\eps(x,t)\omega^\eps(y,t) \dd x \dd y,
\end{equation*}
and that by the definition \eqref{def:psi} of $\psi^\eps$
\begin{equation*}
 \iint G_K(x,y) \omega^\eps(x,t)\omega^\eps(y,t) \dd x \dd y = \int\psi^\eps(x,t)\omega^\eps(x,t)\dd x.
\end{equation*}
Then, using Proposition~\ref{prop:S_K_borné_Lip} and the hypotheses \eqref{eq:a_priori_strong_loc} on the support of $\omega^\eps$ on $[0,T_\eps]$, there exists a constant $C$ such that
\begin{equation*}
 \left| \iint S_{K,\mathcal{U}}(x,y) \omega^\eps(x,t)\omega^\eps(y,t) \dd x \dd y\right| \le C \|\omega^\eps\|_{L^1}^2\leq C.
\end{equation*}
\end{proof}
We now estimate the energy at time 0.
\begin{lemme}\label{lem:E(0)}
We have
\begin{equation*}
 E^\eps(0) = \gamma^2\frac{\sqrt{r_0^2+h^2}}{2\pi h} |\ln\eps| + \mathcal{O}(1)
\end{equation*}
as $\eps \to 0$.
\end{lemme}

\begin{proof}
We recall from \eqref{def:G_K} and \eqref{def:H} that
\begin{equation*}
 \iint G_K(x,y) \omega_0^\eps(x)\omega_0^\eps(y) \dd x \dd y = \iint \frac{\sqrt{|X||Y|}}{2 \pi h} \ln|\mathcal{T}(x)-\mathcal{T}(y)| \omega_0^\eps(x)\omega_0^\eps(y) \dd x \dd y 
\end{equation*}
and thus from Lemma~\ref{lem:E-psiomega} we have that
\begin{multline*}
 E^\eps(0) = -\iint \frac{\sqrt{r_0^2+h^2}}{2\pi h} \ln|\mathcal{T}(x)-\mathcal{T}(y)| \omega_0^\eps(x)\omega_0^\eps(y) \dd x \dd y \\ - \iint \Big( \frac{\sqrt{|X||Y|}}{2 \pi h}-\frac{\sqrt{r_0^2+h^2}}{2\pi h}\Big) \ln|\mathcal{T}(x)-\mathcal{T}(y)| \omega_0^\eps(x)\omega_0^\eps(y) \dd x \dd y + \mathcal{O}(1). 
\end{multline*}
We recall from \eqref{hyp:omega0_reduced} that $\supp \omega_0^\eps \subset B(z_0,\eps)$ so that for every $x,y \in \supp \omega_0^\eps$,
\begin{equation*}
 \Big| \frac{\sqrt{|X||Y|}}{2 \pi h}-\frac{\sqrt{r_0^2+h^2}}{2\pi h}\Big| \le C \eps.
\end{equation*}
Using Lemma~\ref{lem:lnT} two times (in light of Remark~\ref{rem:lnT}), we get that
\begin{align*}
 &\left|\iint \Big( \frac{\sqrt{|X||Y|}}{2 \pi h}-\frac{r_0^2+h^2}{2\pi h}\Big) \ln|\mathcal{T}(x)-\mathcal{T}(y)|\omega_0^\eps(x)\omega_0^\eps(y) \dd x \dd y \right| \le C \eps |\ln\eps|,
\end{align*}
and
\begin{equation*}
 -\iint \ln|\mathcal{T}(x)-\mathcal{T}(y)| \omega_0^\eps(x)\omega_0^\eps(y) \dd x \dd y = \gamma^2 |\ln \eps| + \mathcal{O}(1).
\end{equation*}
We conclude that
\begin{equation*}
 E^\eps(0) = \gamma^2\frac{\sqrt{r_0^2+h^2}}{2\pi h} |\ln\eps| + \mathcal{O}(1).
\end{equation*}
\end{proof}

\begin{lemme}\label{lem:E(t)}
For every $t\le T_\eps$, we have 
\begin{equation*}
 E^\eps(t) = \gamma^2\frac{\sqrt{r_0^2+h^2}}{2\pi h} |\ln\eps| + \mathcal{O}(1)
\end{equation*}
\end{lemme}
\begin{proof}
We shall prove that 
\begin{equation}\label{eq:dE/dt}
 \left|E^\eps(t)-E^\eps(0)\right| = \mathcal{O}(1).
\end{equation}
In order to do so, we write for fixed $\eps$ and $t$:
$$E^\eps(t)=\lim_{\delta \to 0}E^\eps_\delta(t)$$ where 
$$E^\eps_\delta(t)=-\iint \cal G_{K,\cal U,\delta}\omega^\eps(x,t)\omega^\eps(y,t)\dd x \dd y$$
and where $\cal G_{K,\cal U,\delta}$ is obtained from $\cal G_{K,\cal U}$ in the following way:
$$\cal G_{K,\cal U,\delta}= G_{K,\delta}+{S}_{K,U}$$
with
$$G_{K,\delta}(x,y)=
\frac{1}{2\pi} \Big( \det K(x) \det K(y) \Big)^{-1/4} \ln_\delta|\mathcal{T}(x)-\mathcal{T}(y)|
$$
and $\ln_\delta$ is a smooth, even function satisfying $|\nabla \ln_\delta(|x|)\leq C/|x|$ and such that $\ln_\delta|x|=\ln|x|$ on $B(0,\delta)^c$.

Setting 
\begin{equation*}\Psi_{\eps,\delta}(x,t)=\int \cal G_{K,\cal U,\delta}\omega^\eps(y,t)\dd y,
\end{equation*}
we have $\Psi_{\eps,\delta}\in C^1(\mathcal{U}\times[0,T_\eps])$ and therefore by applying Lemma~\ref{lem:deriv} twice we get
\begin{align*}
 -\der{}{t} E^{\eps}_\delta(t) =& 
 \int \omega^\eps(x,t)\left(\partial_t \Psi_{\delta}^\eps(x,t)+\frac{1}{|\ln \eps|}
 \big(v^\eps(x,t)+F^\eps(x,t)\big)\cdot\nabla \Psi_\delta^\eps(x,t)\right)\dd x\\
=& -\frac{1}{|\ln \eps|}\iint 
 \big(v^\eps(y,t)+F^\eps(y,t)\big)\cdot \nabla_y \cal G_{K,U,\delta}(x,y)\omega^\eps(x,t)\omega^\eps(y,t)\dd x \dd y\\
& +\frac{1}{|\ln \eps|}\iint 
\big(v^\eps(x,t)+F^\eps(x,t)\big)\cdot \nabla_x \cal G_{K,U,\delta}(x,y)\omega^\eps(x,t)\omega^\eps(y,t)
\dd x \dd y\\
 =& \frac{1}{|\ln \eps|}\iint 
 \big(v^\eps(x,t)+F^\eps(x,t)\big)\cdot\left( \nabla_y \cal G_{K,U,\delta}(y,x)+ \nabla_x \cal G_{K,U,\delta}(x,y)
 \right)\omega^\eps(x,t)
 \omega^\eps(t,y)\dd x \dd y.
 \end{align*}
Noting that $\cal G_{K,U,\delta}$ is symmetric, we get
\begin{equation*}
 -\der{}{t} E^{\eps}_\delta(t) = \frac{2}{|\ln \eps|}\iint \big(v^\eps(x,t)+F^\eps(x,t)\big)\cdot\nabla_x \cal G_{K,U,\delta}(x,y) \omega^\eps(x,t) \omega^\eps(t,y)\dd x \dd y.
\end{equation*}
Recalling that $v^\eps(x,t)=\int \nabla^\perp \cal G_{K,U}(x,y)\omega^\eps(y,t) \dd y,$ we obtain
\begin{align*}
 -\der{}{t} E^{\eps}_\delta(t) 
=& \frac{2}{|\ln \eps|}\iint 
 \big(v^\eps(x,t)+F^\eps(x,t)\big)\cdot\nabla_x \cal G_{K,U}(x,y)
 \omega^\eps(x,t)
 \omega^\eps(t,y)\dd x \dd y\\
&+ \frac{2}{|\ln \eps|}\iint 
 \big(v^\eps(x,t)+F^\eps(x,t)\big)\cdot
 \left(\nabla_x \cal G_{K,U,\delta}(x,y)-\nabla_x \cal G_{K,U}(x,y)
 \right)\omega^\eps(x,t)
 \omega^\eps(t,y)\dd x \dd y\\
 =&\frac{2}{|\ln \eps|}\iint F^\eps(x,t)\cdot\nabla_x \cal G_{K,U}(x,y)
 \omega^\eps(x,t)
 \omega^\eps(t,y)\dd x \dd y\\
&+ \frac{2}{|\ln \eps|}\iint \omega^\eps(x,t)
 \omega^\eps(t,y)
 \big(v^\eps(x,t)+F^\eps(x,t)\big)\cdot\left(\nabla_x G_{K,\delta}(x,y)-\nabla_x G_{K}(x,y)
 \right)\dd x \dd y\\
 =& \frac{2}{|\ln\eps|} \iint F^\eps(x,t)\cdot \nabla_x G_K(x,y) \omega^\eps(x,t)\omega^\eps(y,t)
 \dd x \dd y \\
 & +\frac{2}{|\ln\eps|} \iint F^\eps(x,t)\cdot
 \nabla_x S_{K,\mathcal{U}}(x,y)
 \omega^\eps(x,t)\omega^\eps(y,t) \dd x \dd y \\
 &+ \frac{2}{|\ln \eps|}\iint 
 \big(v^\eps(x,t)+F^\eps(x,t)\big)\cdot\left(\nabla_x G_{K,\delta}(x,y)-\nabla_x G_{K}(x,y)
 \right)\omega^\eps(x,t)
 \omega^\eps(t,y)\dd x \dd y\\
 :=& A_1 + A_2+A_{3,\delta}.
 \end{align*}

Recalling from \eqref{hyp:F_reduced} that $F^\eps$ is bounded on $\mathcal{A}_{\eta_0}$, and using Proposition~\ref{prop:S_K_borné_Lip}, we have that
$$|A_2| \leq \frac{C}{|\ln\eps|}$$ where $C$ does not depend on $\delta$.

Now we compute $A_1$.
\begin{align*}
 A_1 =& \frac{1}{|\ln\eps|}\iint \Big(\nabla_x G_K(x,y)\cdot F^\eps(x,t)+\nabla_xG_K(y,x)\cdot F^\eps(y,t) \Big)\omega^\eps(x,t)\omega^\eps(y,t)\dd x \dd y \\
 =& \frac{1}{|\ln\eps|}\iint \nabla_x G_K(x,y)\cdot \big(F^\eps(x,t)-F^\eps(y,t)\big) \omega^\eps(x,t)\omega^\eps(y,t)\dd x \dd y \\
 & +\frac{1}{|\ln\eps|}\iint \big(\nabla_xG_K(y,x)+\nabla_xG_K(x,y)\big)\cdot F^\eps(y,t) \omega^\eps(x,t)\omega^\eps(y,t)\dd x \dd y \\
 :=& A_{11} + A_{22}.
\end{align*}
By Lemma~\ref{lem:maj_nabla_GK} and \eqref{hyp:F_reduced}, we observe that 
$$A_{11}\leq \frac{C}{|\ln\eps|}.$$
Recalling from \eqref{hyp:F_reduced} that $F$ is bounded, and using \eqref{eq:decomp_G_K}, we have that
\begin{align*}
 |A_{22}| \le& \frac{C}{|\ln\eps|} \int \Big|\nabla_x H(x,y)+\nabla_x H(y,x)\Big| \Big|\ln|\mathcal{T}(x)-\mathcal{T}(y)|\Big|\omega^\eps(x,t)\omega^\eps(y,t)\dd x \dd y \\
 & + \frac{C}{|\ln\eps|}\int H(x,y) \frac{|\Diff \mathcal{T}(x)-\Diff \mathcal{T}(y)|}{\big|\mathcal{T}(x)-\mathcal{T}(y) \big|}\omega^\eps(x,t)\omega^\eps(y,t)\dd x \dd y.
\end{align*}
By Remark~\ref{rem:lnT} and since $H$ and $\mathcal{T}$ are smooth, we obtain that
\begin{align*}
 |A_{22}| & \le \frac{C}{|\ln\eps|} \int \Big|\ln|\mathcal{T}(x)-\mathcal{T}(y)|\Big|\omega^\eps(x,t)\omega^\eps(y,t)\dd x \dd y + \frac{C}{|\ln\eps|} \int \omega^\eps(x,t)\omega^\eps(y,t)\dd x \dd y \\
 & \le C + \frac{C}{|\ln\eps|} \\
 &\leq C
\end{align*}
uniformly with respect to $\delta$.

Finally, we observe that the support of $\nabla_x G_{K,\delta}(x,y)-\nabla_x G_{K}(x,y)$ is included in the set of $x,y$ such that $|x-y|\leq C\delta$ by \eqref{eq:Diff_T_Lip}. Thus using Lemma~\ref{lem:maj_nabla_GK} and using the definition of $\ln_\delta$ we estimate
\begin{align*}
 |A_{3,\delta}|&\leq \frac{C}{|\ln \eps|}\iint_{|x-y|\leq C\delta} 
 \big(|v^\eps(x,t)|+|F^\eps(x,t)|\big)|x-y|^{-1}\omega^\eps(x,t)
 \omega^\eps(t,y)\dd x \dd y\\
 &\leq \frac{C}{|\ln \eps|}\left(\|v^\eps\|_{L^\infty}+\|F^\eps\|_{L^\infty}\right)\|
 \omega^\eps\|_{L^\infty}\delta.
 \end{align*}
The latter diverges as $\eps\to 0$ but letting $\delta$ tend to zero for fixed $\eps$ we obtain
\begin{equation*}
 \limsup_{\delta\to 0} |E^{\eps}_\delta(t)-E^{\eps}_\delta(0)|\leq C T,
\end{equation*}
and the conclusion \eqref{eq:dE/dt} follows.

Recalling Lemma~\ref{lem:E(0)}, we obtain the desired result of the lemma.
\end{proof}

An interesting corollary of Lemma~\ref{lem:E(t)} is the following estimate on the local energy $\psi^\eps$, defined at \eqref{def:psi}.
\begin{coro}\label{lem:variance}
We have for all $t\le T_\eps$ 
\begin{equation*}
 \int \left| \gamma \psi^\eps(x,t) - \int \psi^\eps(y,t) \omega^\eps(y,t)\dd y \right|^2 \omega^\eps(x,t)\dd x = \mathcal{O}\big(|\ln \eps|\big).
\end{equation*}
\end{coro}

\begin{proof}
We start by noticing that developing the square,
\begin{align*}
 \int \Bigg| \gamma \psi^\eps(x,t) - &\int \psi^\eps(y,t) \omega^\eps(y,t)\dd y \Bigg|^2 \omega^\eps(x,t)\dd x = \\
=& \gamma^2\int (\psi^\eps)^2(x,t) \omega^\eps(x,t) \dd x - \gamma\left(\int \psi^\eps(x,t)\omega^\eps(x,t) \dd x\right)^2.
\end{align*}
Now recalling Lemmas~\ref{lem:E-psiomega} and \ref{lem:E(t)}, we have that
\begin{equation*}
 \left(\int \psi^\eps(x,t)\omega^\eps(x,t) \dd x\right)^2 = \big(E^\eps(t) + \mathcal{O}(1) \big)^2 = \gamma^4\frac{r_0^2+h^2}{4\pi^2 h^2} |\ln\eps|^2 + \mathcal{O}\big(|\ln\eps|\big).
\end{equation*}
Using Lemma~\ref{lem:encadrement_psi}, we obtain that
\begin{align*}
 \gamma^2\int (\psi^\eps)^2(x,t) \omega^\eps(x,t) \dd x & \le \gamma^2\int \left(\gamma\frac{|X|}{2\pi h}|\ln\eps| + \mathcal{O}(1)\right)^2 \omega^\eps(x,t) \dd x \\
 & \le \frac{\gamma^4}{4\pi^2 h^2}|\ln\eps|^2 \int |X|^2 \omega^\eps(x,t) \dd x + \mathcal{O}(|\ln\eps|).
\end{align*}
Therefore,
\begin{multline*}
 \int \left| \gamma \psi^\eps(x,t) - \int \psi^\eps(y,t) \omega^\eps(y,t)\dd y \right|^2 \omega^\eps(x,t)\dd x \le \frac{\gamma^4}{4\pi^2 h^2} |\ln\eps|^2 \left( \int |X|^2 \omega^\eps(x,t)\dd x-\gamma(r_0^2+h^2) \right) \\ +\mathcal{O}(|\ln\eps|).
\end{multline*}
Using Lemma~\ref{lem:J} for $k=2$, we have that
\begin{equation*}
 \int |X|^2 \omega^\eps(x,t) \dd x -\gamma(r_0^2+h^2)= J_2^\eps(t)-\gamma r_0^2 = \mathcal{O}\left(\frac{1}{|\ln\eps|}\right).
\end{equation*}

In conclusion,
\begin{equation*}
 \int \left| \gamma \psi^\eps(x,t) - \int \psi^\eps(y,t) \omega^\eps(y,t)\dd y \right|^2 \omega^\eps(x,t)\dd x \le \mathcal{O}(|\ln\eps|).
\end{equation*}
\end{proof}

\subsection{Estimates on first and second vorticity moments}
Let us introduce the center of mass
\begin{equation*}
 b^\eps(t) = \frac{1}{\gamma}\int x \omega^\eps(x,t) \dd x,
\end{equation*}
and we denote $|B^\eps(t)| = \sqrt{|b^\eps(t)|^2 + h^2}$. Let us also define the center of inertia around $b^\eps$:
\begin{equation}\label{def:I}
 I^\eps(t) = \int |x-b^\eps(t)|^2 \omega^\eps(x,t)\dd x.
\end{equation}
We start by computing the derivative of $b^\eps$ using Lemma~\ref{lem:deriv},
\begin{align*}
 \der{}{t} b^\eps(t) = \frac{1}{\gamma|\ln\eps|}\int \big(v^\eps(x,t) + F^\eps(x,t) \big) \omega^\eps(x,t) \dd x.
\end{align*}
We then establish the following lemma.
\begin{lemme}\label{lem:b} For any $t \le T_\eps$, we have that
\begin{equation*}
 \der{}{t} b^\eps(t) =- \gamma\frac{\sqrt{r_0^2+h^2}}{4\pi h} \frac{b^\eps(t)^\perp}{|B^\eps(t)|^2} + \mathcal{O}\left(\frac{1}{\sqrt{|\ln\eps|}}\right) + \mathcal{O}\left(\sqrt{I^\eps(t)}\right).
\end{equation*}
\end{lemme}

\begin{proof}
Recalling the decomposition~\eqref{eq:decomp:v}, we have that
\begin{equation}\label{eq:decomp_b}
 \der{}{t} b^\eps(t) = \frac{1}{\gamma|\ln\eps|}\int \big(v_K^\eps(x,t) + v_L^\eps(x,t) + v_R^\eps(x,t) + F^\eps(x,t) \big) \omega^\eps(x,t) \dd x.
\end{equation}
By \eqref{hyp:F_reduced}, Proposition~\ref{prop:S_K_borné_Lip} and hypotheses \eqref{hyp:omega0_reduced}, the functions $v_R^\eps$ and $F^\eps$ are bounded in space for $x \in \mathcal{A}_{\eta_0}$ uniformly in time $t \in [0,T_\eps]$ so that
\begin{equation}\label{eq:est1}
 \frac{1}{\gamma|\ln\eps|}\int \big(v_R^\eps(x,t) + F^\eps(x,t) \big) \omega^\eps(x,t) \dd x = \mathcal{O}\left( \frac{1}{|\ln\eps|}\right).
\end{equation}
We turn to the term containing $v_K^\eps$. Recalling its definition \eqref{def:v_K}, that $H(x,y) = H(y,x)$ and \eqref{eq:Diff_T_Lip}, we have
\begin{align*}
 \left| \int v_K^\eps(x,t)\omega^\eps(x,t)\dd x\right|& = \left|\iint H(x,y) \Diff \mathcal{T}(x) \frac{\mathcal{T}(x)-\mathcal{T}(y)}{\big|\mathcal{T}(x)-\mathcal{T}(y) \big|^2}\omega^\eps(x,t)\omega^\eps(y,t)\dd x \dd y \right| \\
 & = \left|\iint H(x,y) \big(\Diff \mathcal{T}(x) - \Diff \mathcal{T}(y)\big) \frac{\mathcal{T}(x)-\mathcal{T}(y)}{2\big|\mathcal{T}(x)-\mathcal{T}(y) \big|^2}\omega^\eps(x,t)\omega^\eps(y,t)\dd x \dd y \right| \\
 & \le C\iint |H(x,y)| \omega^\eps(x,t)\omega^\eps(y,t)\dd x \dd y
\end{align*}
and thus
\begin{equation}\label{eq:est2}
 \frac{1}{\gamma|\ln\eps|}\int v_K^\eps(x,t)\omega^\eps(x,t)\dd x = \mathcal{O}\left(\frac{1}{|\ln\eps|}\right).
\end{equation}
We now estimate the term involving $v_L^\eps$ which is the leading term of the movement. Recalling \eqref{eq:v_L-psi}, we have that
\begin{align*}
 \int v_L^\eps(x,t)\omega^\eps(x,t)\dd x & = \int \frac{x^\perp}{2|X|^2}\psi^\eps(x,t)\omega^\eps(x,t)\dd x.
\end{align*}
Now we compute, using Corollary~\ref{lem:variance} and the Cauchy-Schwarz inequality, that for every $t \le T_\eps$,
\begin{multline*}
 \left|\int v_L^\eps(x,t)\omega^\eps(x,t)\dd x - \frac{1}{\gamma}\int \frac{x^\perp}{2|X|^2}\int \psi^\eps(y,t)\omega^\eps(y,t)\dd y \, \omega^\eps(x,t)\dd x\right| \\ = \frac{1}{\gamma}\left|\int \frac{x^\perp}{2|X|^2} \left(\gamma\psi^\eps(x,t) - \int\psi^\eps(y,t)\omega^\eps(y,t)\dd y\right)\omega^\eps(x,t)\dd x \right|
 = \mathcal{O}\left( \sqrt{|\ln\eps|}\right).
\end{multline*}
Using Lemmas~\ref{lem:E-psiomega} and \ref{lem:E(t)}, we thus infer that
\begin{align*}
 \frac{1}{\gamma}\int \frac{x^\perp}{2|X|^2}\int \psi^\eps(y,t)\omega^\eps(y,t)\dd y \, \omega^\eps(x,t)\dd x & = \frac{-E^\eps(t)}{\gamma}\int \frac{x^\perp}{2|X|^2}\omega^\eps(x,t)\dd x + \mathcal{O}(1) \\
 & = -\frac{\gamma\sqrt{r_0^2 + h^2}}{2\pi h}|\ln \eps|\int \frac{x^\perp}{2|X|^2}\omega^\eps(x,t)\dd x + \mathcal{O}(1).
\end{align*}
We conclude from the two previous relations that
\begin{equation}\label{eq:est3}
 \frac{1}{\gamma|\ln\eps|}\int v_L^\eps(x,t)\omega^\eps(x,t)\dd x = -\frac{\sqrt{r_0^2+h^2}}{4\pi h} \int \frac{x^\perp}{|X|^2}\omega^\eps(x,t)\dd x + \mathcal{O}\left(\frac{1}{\sqrt{|\ln\eps|}}\right).
\end{equation}
At this point, gathering the estimates \eqref{eq:decomp_b}, \eqref{eq:est1}, \eqref{eq:est2} and \eqref{eq:est3}, we have proved that
\begin{equation*}
 \der{}{t} b^\eps(t) = -\frac{\sqrt{r_0^2+h^2}}{4\pi h} \int \frac{x^\perp}{|X|^2}\omega^\eps(x,t)\dd x + \mathcal{O}\left(\frac{1}{\sqrt{|\ln\eps|}}\right) .
\end{equation*}
There remains to evaluate
\begin{equation*}
 \int \frac{x^\perp}{|X|^2}\omega^\eps(x,t)\dd x - \gamma\frac{\big(b^\eps(t)\big)^\perp}{|B^\eps(t)|^2} = \int \left( \frac{x^\perp}{|X|^2}-\frac{\big(b^\eps(t)\big)^\perp}{|B^\eps(t)|^2}\right)\omega^\eps(x,t)\dd x. 
\end{equation*}
Since the function $x \mapsto \frac{x^\perp}{|X|^2}$ is smooth on $\mathcal{A}_{\eta_0}$, there exists a constant $C$ such that
\begin{equation*}
 \left|\int \frac{x^\perp}{|X|^2}\omega^\eps(x,t)\dd x - \gamma\frac{\big(b^\eps(t)\big)^\perp}{|B^\eps(t)|^2} \right| \le \int C|x-b^\eps(t)| \omega^\eps(x,t)\dd x \le C \sqrt{I^\eps(t)}.
\end{equation*}
In the end, we proved that
\begin{equation*}
 \der{}{t} b^\eps(t) =- \gamma\frac{\sqrt{r_0^2+h^2}}{4\pi h} \frac{b^\eps(t)^\perp}{|B^\eps(t)|^2} + \mathcal{O}\left(\frac{1}{\sqrt{|\ln\eps|}}\right) + \mathcal{O}\left(\sqrt{I^\eps(t)}\right).
\end{equation*}
\end{proof}

We now turn to the study of the moment of inertia $I^\eps$.
\begin{lemme}\label{lem:dI/dt}
For every $t \le T_\eps$, we have that
\begin{equation*}
 \der{}{t} I^\eps(t) \le C \Big(I^\eps(t)+ \frac{1}{|\ln\eps|}\Big).
\end{equation*}
\end{lemme}
\begin{proof}
We compute from \eqref{def:I} and Lemma~\ref{lem:deriv}:
\begin{align*}
 \der{}{t} I^\eps(t) = \frac{2}{|\ln\eps|}\int (x-b^\eps(t))\cdot \big( v^\eps(x,t) + F^\eps(x,t) \big)\omega^\eps(x,t)\dd x.
\end{align*}
Recalling the decomposition \eqref{eq:decomp:v},
\begin{equation}\label{est11}
 \der{}{t} I^\eps(t) = \frac{2}{|\ln\eps|}\int (x-b^\eps(t))\cdot \big( v_K^\eps(x,t) + v_L^\eps(x,t) + v_R^\eps(x,t) + F^\eps(x,t) \big)\omega^\eps(x,t)\dd x.
\end{equation}
Similarly to what we did in the proof of Lemma~\ref{lem:b}, we use the boundedness of $v_R^\eps$ and $F^\eps$ to get, by the Cauchy-Schwarz inequality, that
\begin{equation}\label{est12}
\left|\int (x-b^\eps(t))\cdot\big( v_R^\eps(x,t) + F^\eps(x,t) \big)\omega^\eps(x,t)\dd x \right|
 \le C\Bigg(\int |x-b^\eps(t)|^2 \omega^\eps(x,t)\dd x\Bigg)^{1/2}=C\sqrt{I^\eps(t)}. 
\end{equation}
Next, we compute
\begin{multline*}
 \int (x-b^\eps(t))\cdot v_K^\eps(x,t)\omega^\eps(x,t)\dd x
 \\ =- \iint (x-b^\eps(t))^\perp\cdot H(x,y) \Diff \mathcal{T}(x) \frac{\mathcal{T}(x)-\mathcal{T}(y)}{\big|\mathcal{T}(x)-\mathcal{T}(y) \big|^2}\omega^\eps(y,t)\omega^\eps(x,t) \dd x \dd y.
\end{multline*}
We temporarily denote $z = \mathcal{T}(x)-\mathcal{T}(y)$ and symmetrize the expression to get that
\begin{align*}
 & \left|\int (x-b^\eps(t))\cdot v_K^\eps(x,t)\omega^\eps(x,t)\dd x\right|
 \\ & = \frac{1}{2}\left|\iint H(x,y) \left[(x-b^\eps(t))^\perp\cdot \left(\Diff \mathcal{T}(x)\frac{z}{|z|^2}\right)- (y-b^\eps(t))^\perp\cdot\left( \Diff \mathcal{T}(y)\frac{z}{|z|^2}\right)\right]\omega^\eps(y,t)\omega^\eps(x,t) \dd x \dd y \right| \\
 & = \frac{1}{2}\Bigg|\iint\frac{H(x,y)}{|z|^2} \Big((x-b^\eps(t))^\perp\cdot \left(\Diff \mathcal{T}(x)z\right)-(x-b^\eps(t))^\perp\cdot \left(\Diff \mathcal{T}(y)z \right)\\
 & \hspace{5cm} + (x-b^\eps(t))^\perp\cdot \Diff \mathcal{T}(y)z - (y-b^\eps(t))^\perp\cdot \Diff \mathcal{T}(y)z\Big)\omega^\eps(y,t)\omega^\eps(x,t) \dd x \dd y \Bigg| \\
 & \le C\iint |H(x,y)||x-b^\eps(t)| \omega^\eps(y,t)\omega^\eps(x,t) \dd x \dd y +\frac{1}{2}\iint |H(x,y)| \frac{|x-y|}{|z|^2}|\Diff \mathcal{T}(y)z|\omega^\eps(y,t)\omega^\eps(x,t) \dd x \dd y
\end{align*}
and obtain by the Cauchy-Scwharz inequality and since $H$ is bounded that
\begin{equation}\label{est13}
 \left|\int (x-b^\eps(t))\cdot v_K^\eps(x,t)\omega^\eps(x,t)\dd x\right| \le C\sqrt{I^\eps(t)} + \mathcal{O}\left( 1\right).
\end{equation}

We now compute
\begin{align*}
 \int (x-b^\eps(t))&\cdot v_L^\eps(x,t) \omega^\eps(x,t)\dd x 
 = \int (x-b^\eps(t))\cdot\frac{x^\perp}{2|X|^2} \psi^\eps(x,t)\omega^\eps(x,t)\dd x \\
 & = \frac{1}{\gamma}\int (x-b^\eps(t))\cdot\frac{x^\perp}{2|X|^2} \left(\gamma\psi^\eps(x,t)-\int\psi^\eps(y,t)\omega^\eps(y,t)\dd y \right)\omega^\eps(x,t)\dd x \\
 & \qquad + \frac{1}{\gamma}\int (x-b^\eps(t))\cdot\frac{x^\perp}{2|X|^2} \int\psi^\eps(y,t)\omega^\eps(y,t)\dd y \omega^\eps(x,t)\dd x.
\end{align*}
Using the Cauchy-Schwarz inequality and Corollary~\ref{lem:variance}, we estimate the first term of the right hand side of the previous equality:
\begin{equation*}
 \left|\frac{1}{\gamma}\int (x-b^\eps(t))\cdot\frac{x^\perp}{2|X|^2} \left(\gamma\psi^\eps(x,t)-\int\psi^\eps(y,t)\omega^\eps(y,t)\dd y \right)\omega^\eps(x,t)\dd x \right|\le C \sqrt{|\ln\eps| I^\eps(t)}.
\end{equation*}
Using the relation with the energy Lemmas~\ref{lem:E-psiomega} and~\ref{lem:E(t)}, we estimate the second term as:
\[
\Bigg| \frac{1}{\gamma}\int (x-b^\eps(t))\cdot\frac{x^\perp}{2|X|^2} \int\psi^\eps(y,t)\omega^\eps(y,t)\dd y \omega^\eps(x,t)\dd x\Bigg| \leq C |\ln \eps| \Bigg|\int (x-b^\eps(t))\cdot\frac{x^\perp}{2|X|^2} \omega^\eps(x,t)\dd x \Bigg|+C\sqrt{I^\eps(t)}.
\]
Moreover,
\begin{multline}\label{est14-1}
 \int (x-b^\eps(t))\cdot\frac{x^\perp}{|X|^2} \omega^\eps(x,t)\dd x \\ = \int (x-b^\eps(t))\cdot\frac{x^\perp}{|B^\eps(t)|^2} \omega^\eps(x,t)\dd x + \int (x-b^\eps(t))\cdot x^\perp\left(\frac{1}{|X|^2} - \frac{1}{|B^\eps(t)|^2} \right) \omega^\eps(x,t)\dd x,
\end{multline}
and we then notice that
\begin{equation*}
 \int (x-b^\eps(t))\cdot\frac{x^\perp}{|B^\eps(t)|^2} \omega^\eps(x,t)\dd x = -\frac{b^\eps(t)}{|B^\eps(t)|^2}\cdot\int x^\perp\omega^\eps(x,t)\dd x = -\frac{b^\eps(t)}{|B^\eps(t)|^2}\cdot \gamma\big(b^\eps(t)\big)^\perp(t)=0.
\end{equation*}
We then use the fact that there exists a constant $C$ such that
\begin{equation*}
 \left|\frac{1}{|X|^2} - \frac{1}{|B^\eps(t)|^2} \right| \le C|x(t)-b^\eps(t)|
\end{equation*}
to conclude from \eqref{est14-1} that
\begin{equation*}
 \Bigg| \int (x-b^\eps(t))\cdot\frac{x^\perp}{|X|^2} \omega^\eps(x,t)\dd x \Bigg| \le C I^\eps(t).
\end{equation*}
We have proved that
\begin{equation}\label{est14}
 \left|\int (x-b^\eps(t))\cdot v_L^\eps(x,t)\omega^\eps(x,t)\dd x \right| \le C |\ln \eps| I^\eps(t) + C\sqrt{|\ln\eps|I^\eps(t)}.
\end{equation}
Gathering \eqref{est11}, \eqref{est12}, \eqref{est13} and \eqref{est14}, we conclude that there exists a constant $C$ independent of $\eps$ such that
\begin{equation*}
 \der{}{t} I^\eps(t) \le \frac{C}{|\ln\eps|}\left( \sqrt{I^\eps(t)} + 1 + |\ln \eps| I^\eps(t) + \sqrt{|\ln \eps| I^\eps(t)}\right),
\end{equation*}
which implies that
\begin{equation*}
 \der{}{t} I^\eps(t) \le C \left( I_\eps + \frac{1}{|\ln\eps|}\right).
\end{equation*}
\end{proof}

\subsection{Weak localization}

We are now in position to prove the first part of Theorem~\ref{theo:reduit}. First we get that the center of mass of the vorticity remains close to $z$.
\begin{lemme}\label{eq:I} For every $t \in [0,T_\eps]$, we have that
\begin{equation*}
 I^\eps(t) \le \frac{C}{|\ln\eps|},
\end{equation*}
and
\begin{equation*}
|b^\eps(t)-z(t)| \le \frac{C}{\sqrt{|\ln\eps|}}.
\end{equation*}
\end{lemme}
\begin{proof}
From Lemma~\ref{lem:dI/dt}, by Gronwall's inequality, we have that for all $t \in [0,T_\eps]$,
\begin{equation*}
 I^\eps(t) \le \left(I^\eps(0)+\frac{Ct}{|\ln\eps|}\right) e^{Ct}
\end{equation*}
and thus there exists a constant $C$ depending on $T$ such that for every $\eps > 0$ and every $t \in [0,T_\eps]$,
\begin{equation*}
 I^\eps(t) \le \frac{C}{|\ln\eps|},
\end{equation*}
which proves the first estimate of Lemma~\ref{eq:I}.

We now plug this into the result of Lemma~\ref{lem:b} to obtain that
\begin{equation*}
 \der{}{t} b^\eps(t) =- \gamma\frac{\sqrt{r_0^2+h^2}}{4\pi h} \frac{b^\eps(t)^\perp}{|B^\eps(t)|^2} + \mathcal{O}\left(\frac{1}{\sqrt{|\ln\eps|}}\right).
\end{equation*}
We define
\begin{equation*}
 f(x) = -\gamma\frac{\sqrt{r_0^2+h^2}}{4\pi h}\frac{x^\perp}{|X|^2}.
\end{equation*}
We start by checking that the function $z(t)= \tilde{R}_{t \nu } z_{0}$ with
$ \nu = -\frac{\gamma}{4\pi h |Z_0|}$ is a solution of $z'(t) = f(z(t))$.
Indeed,
\begin{equation*}
 z'(t) = \nu z^\perp(t) =- \frac{\gamma}{4\pi h \sqrt{h^2+r_0^2}}z^\perp(t) = -\gamma\frac{\sqrt{r_0^2+h^2}}{4\pi h}\frac{z^\perp(t)}{|Z(t)|^2} = f(z(t)).
\end{equation*}
Since $f$ is a Lipschitz function, and $|b^\eps(0)-z_0| \le \eps$, we now use Gronwall's inequality, given in Lemma~\ref{lem:gronwall} with $g(t)=C/\sqrt{|\ln\eps|}$, to obtain that there exists $\kappa$ such that for every $t \in [0,T_\eps]$,
\begin{equation*}
 |b^\eps(t)-z(t)| \le \left( \frac{tC}{\sqrt{|\ln\eps|}} + \eps\right)e^{\kappa t}.
\end{equation*}
Therefore, there exists $C$ depending on $T$ but not on $\eps$ such that for every $\eps$ small enough and for every $t \in [0,T_\eps]$,
\begin{equation*}
 |b^\eps(t)-z(t)| \le \frac{C}{\sqrt{|\ln \eps|}}.
\end{equation*}
\end{proof}

Finally, we estimate the amount of vorticity outside a small disk centered at $z$.

\begin{prop} Let $r_\eps = \left(\frac{\ln |\ln\eps|}{|\ln \eps|}\right)^{1/2}.$ Then there exists $C$ such that for every $t \in [0,T_\eps]$,
\begin{equation*}
 \int_{\R^2 \setminus B(z(t),r_\eps)} \omega^\eps(x,t) \dd x \le \frac{C}{\ln |\ln \eps|}.
\end{equation*}
\end{prop}
\begin{proof}
We observe that
\begin{align*}
 \int_{\cal U \setminus B(z(t),r_\eps)} \omega^\eps(x,t) \dd x & \leq \int_{\cal U \setminus B(z(t),r_\eps)} \frac{|x-z(t)|^2}{r_\eps^2}\omega^\eps(x,t) \dd x \\
 & \le \frac2{r_\eps^2}\int_{\cal U \setminus B(z(t),r_\eps)}( |x-b^\eps(t)|^2+|b^\eps(t)-z(t)|^2) \omega^\eps(x,t) \dd x \\
 & \le \frac{C}{r_\eps^2 |\ln \eps|},
\end{align*}
which ends the proof.
\end{proof}

These two lemmas allow us to conclude the part $(i)$ of Theorem~\ref{theo:reduit}. To manage to prove that $T_\eps=T$, namely that the vorticity remains supported in annulus, we prove in the next section the strong localization in the radial direction.

\subsection{Strong localization}

For $t \in [0,T_\eps]$, let us define the mass of vorticity outside the annulus of thickness $\eta$ (see \eqref{def:Aeta}) by
\begin{equation*}
 m_t(\eta) := \int_{\cal A_\eta} \omega^\eps(y,t) \dd y 
\end{equation*}
and let $s \mapsto x_t(s)$ be the Lagrangian trajectory passing by $x_0$ at time $t$, i.e. the solution of
\begin{equation*}
 \left\{\begin{aligned}
& \der{}{s} x_t(s) =\frac1{|\ln\eps|}\Big( v^\eps(x_t(s),s) + F^\eps(x_t(s),s)\Big), \\
& x_t(t) = x_0,
\end{aligned}\right.
\end{equation*}
with $x_0$ being such that 
\begin{equation*}
 \Big||x_0| - r_0\Big| = R_t := \max \Big\{ \Big||x|-r_0\Big| \, , \, x \in \supp \omega^\eps(t)\Big\}.
\end{equation*}

\begin{lemme}\label{lemma:Rt}
\begin{equation*}
 \der{}{s} \Big| |x_t(s)| - r_0\Big|(t) \le \frac{C}{|\ln\eps|} \left( 1 + \frac{1}{R_t} + \frac{\sqrt{m_t(R_t/2)}}{\eps}\right). 
\end{equation*}
\end{lemme}
\begin{proof}
We use that $x\cdot v_L(x,t)\equiv 0$ to write
\begin{align*}
 \der{}{s} \Big| |x_t(s)| - r_0\Big|(t) & = \frac{x_t(t)}{|x_t(t)|}\cdot \left[\frac{1}{|\ln\eps|} (v^\eps+F^\eps)(x_t(t),t)\right] \frac{|x_t(t)| - r_0}{\big||x_t(t)| - r_0\big|}\\
 & \le \frac{1}{|\ln\eps|} \Big| v_K ^\eps(x_0,t) + v_R^\eps(x_0,t)+F^\eps(x_0,t)\Big|\\
 &\le \frac{C}{\ln\eps} \left( 1 + \int \frac{\omega^\eps(y,t)}{\big|\mathcal{T}(x_0)-\mathcal{T}(y)\big|} \dd y\right) \\
 &\le \frac{C}{\ln\eps} \left( 1 + \int \frac{\omega^\eps(y,t)}{|x_0-y|}\dd y\right). \\
\end{align*}
We now split the integral on $\mathcal{A}_{R_t/2}$ and $\cal U \setminus \mathcal{A}_{R_t/2}.$ If $y \in \mathcal{A}_{R_t/2}$, then $|x_0-y| \ge R_t/2$, and thus
\begin{equation*}
 \int_{\mathcal{A}_{R_t/2}} \frac{\omega^\eps(y,s)}{|x_0-y|}\dd y \le \frac{C}{R_t}.
\end{equation*}
On the complement, we use the following property: for every $\Omega \subset \R^2$
\begin{equation*}
 \int_\Omega \frac{\omega^\eps(y,s)}{|x_0-y|}\dd y \le C \| \omega^\eps(s) \|_{L^1(\Omega)}^{\frac{1}{2}} \| \omega^\eps(s)\|_{L^\infty(\Omega)}^{\frac{1}{2}},
\end{equation*}
where $C$ is independent of $\Omega$, to obtain that
\begin{equation*}
 \int_{\cal U\setminus \mathcal{A}_{R_t/2}} \frac{\omega^\eps(y,s)}{|x_0-y|}\dd y \le \frac{C}{\eps}\sqrt{m_t(R_t/2)}.
\end{equation*}
\end{proof}

\begin{lemme}\label{lemma-step}
For any $\ell > 0$ and $\kappa \in (0,\frac{1}{4})$,
\begin{equation*}
 \lim_{\eps \to 0} \eps^{-\ell} m_t\left(\frac{1}{|\ln\eps|^\kappa}\right) = 0.
\end{equation*}
\end{lemme}
\begin{proof}
We use the same plan as in the proof of Lemma 7.3 in \cite{Hientzsch_Lacave_Miot_2022_Dynamics_of_PV_for_the_lake_eq}. Let
\begin{equation*}
 \mu_t(R,\eta) := \int \Big( 1- W_{R,\eta}\big(|y|-r_0\big) \Big) \omega^\eps(y,t)\dd t
\end{equation*}
where $ W_{R,\eta}$ is a smooth non-negative function from $\R$ to $\R$ such that
\begin{equation*}
 W_{R,\eta}(s) =\begin{cases} 1 & \text{if } |s|\le R \\ 0 & \text{if } |s| \ge R+\eta\end{cases}
\end{equation*}
satisfying $|W_{R,\eta}'|\le C /\eta$ and $|W_{R,\eta}''|\le C /\eta^2$. In particular, we have that
\begin{equation*}
 \mu_t(R,\eta) \le m_t(R) \le \mu_t(R-\eta,\eta).
\end{equation*}
We now compute
\begin{align*}
 \der{}{t} \mu_t(R,\eta) & = -\frac{1}{|\ln\eps|} \int \nabla \Big( W_{R,\eta} \big(|y| - r_0\big) \Big) \cdot (v^\eps + F^\eps)(y,t) \omega^\eps(y,t)\dd y \\
 & = -\frac{1}{|\ln\eps|}\int W_{R,\eta}' \big(|y|-r_0\big)\frac{y}{|y|}\cdot (v^\eps + F^\eps)(y,t)\omega^\eps(y,t)\dd y.
\end{align*}
Noticing from \eqref{def:v_L} that $v_L^\eps(y,t) \cdot y = 0$, and recalling the decomposition \eqref{eq:decomp:v}, we have that
\begin{align*}
 \der{}{t} \mu_t(R,\eta) & = -\frac{1}{|\ln\eps|} \int W_{R,\eta}' \big(|y|-r_0\big)\frac{y}{|y|}\cdot (v_R^\eps + F^\eps)(y,t) \omega^\eps(y,t)\dd y \\ & \qquad -\frac{1}{|\ln\eps|} \int W_{R,\eta}' \big(|y|-r_0\big)\frac{y^\perp}{|y|}\cdot v_K^\eps(y,t)^\perp\omega^\eps(y,t) \dd y \\ & := A_1 + A_2.
\end{align*}
Recalling that $v_R^\eps$ and $F^\eps$ are bounded, that $|W'_{R,\eta}|\leq C/\eta$ and that $W'(s) = 0$ outside of the annulus $R \le |s| \le R+\eta $, we have that
\begin{equation*}
 |A_1| \le \frac{C}{\eta|\ln\eps|} \int_{R \le \big||y|-r_0 \big|\le R+\eta} \omega^\eps(y,t)\dd y \le \frac{C m_t(R)}{\eta|\ln\eps|}.
\end{equation*}
In order to deal with $A_2$, recalling the definition of $v_K^\eps$ \eqref{def:v_K}, we first notice by symmetrization that
\begin{equation*}
 A_2 = \frac{1}{4\pi|\ln\eps|} \iint f(x,y,t)\dd x \dd y
\end{equation*}
by letting
\begin{equation*}
 f(x,y,t) = H(x,y) \left( W_{R,\eta}' \big(|y|-r_0\big)\chi(y) - W_{R,\eta}' \big(|x|-r_0\big)\chi(x) \right) \cdot \frac{\mathcal{T}(y)-\mathcal{T}(x)}{\big|\mathcal{T}(y)-\mathcal{T}(x) \big|^2} \omega^\eps(y,t)\omega^\eps(x,t)
\end{equation*}
and
\begin{equation*}
 \chi(y) = \frac{\Diff \mathcal{T}(y) y^\perp}{|y|},
\end{equation*}
since $\Diff \mathcal{T}$ is a symmetric matrix. We observe that $f(x,y,t) = 0$ for every $(x,y) \in (\mathcal{A}_{R})^2.$ Then, for any $\alpha >0$ such that $R > 2 \eta^\alpha$, we have that
\begin{multline*}
 \iint f(x,y,t)\dd x \dd y = 2\iint_{\mathcal{A}_R^c \times \mathcal{A}_{R-\eta^\alpha}} f(x,y,t)\dd x \dd y + 2\iint_{\mathcal{A}_R^c \times \mathcal{A}_{R-\eta^\alpha}^c} f(x,y,t)\dd x \dd y \\ - \iint_{\mathcal{A}_R^c \times \mathcal{A}_{R}^c} f(x,y,t)\dd x \dd y.
\end{multline*}

When $(x,y) \in \mathcal{A}_R^c \times \mathcal{A}_{R-\eta^\alpha}$, then $|x-y| \ge \eta^\alpha$. In particular,
\begin{equation*}
 \frac{|\mathcal{T}(y)-\mathcal{T}(x)|}{\big|\mathcal{T}(y)-\mathcal{T}(x) \big|^2} \le \frac{C}{\eta^\alpha}.
\end{equation*}
Using in addition that $|W'_{R,\eta}|\le C/\eta$ and that $H$ and $\chi$ are bounded, we have that
\begin{equation*}
 \left|\iint_{\mathcal{A}_R^c \times \mathcal{A}_{R-\eta^\alpha}} f(x,y,t)\dd x \dd y\right| \le \frac{C}{\eta^{1+\alpha}} \iint_{\mathcal{A}_R^c \times \mathcal{A}_{R-\eta^\alpha}}\omega^\eps(y,t)\omega^\eps(x,t)\dd x \dd y \le C\frac{m_t(R)}{\eta^{1+\alpha}}.
\end{equation*}
For the other two terms, we use the fact that
\begin{equation*}
 \big| W_{R,\eta}' \big(|y|-r_0\big)\chi(y) - W_{R,\eta}' \big(|x|-r_0\big)\chi(x) \big| \le C\left(\frac{1}{\eta^2} + \frac{1}{\eta}\right)|x-y|
\end{equation*}
since $W_{R,\eta}'' \le C/\eta^2$, $W_{R,\eta}' \le C/\eta$ and $\chi'$ is bounded. Therefore, assuming that $\eta\le 1$ and $R-\eta^\alpha > R/2$, we have
\begin{align*}
 \left|2\iint_{\mathcal{A}_R^c \times \mathcal{A}_{R-\eta^\alpha}^c} f(x,y,t)\dd x \dd y\right|+ \left|\iint_{\mathcal{A}_R^c \times \mathcal{A}_{R}^c} f(x,y,t)\dd x \dd y\right| & \le \frac{C}{\eta^2} \iint_{\mathcal{A}_R^c \times \mathcal{A}_{R/2}^c}\omega^\eps(y,t)\omega^\eps(x,t)\dd x \dd y \\
 & \le \frac{Cm_{t}(R)}{\eta^2 R^2}\int_{\mathcal{A}_{R/2}^c}\big||x|-r_0\big|^2\omega^\eps(x,t)\dd x.
\end{align*}
We now recall the definition of $J_k^\eps$ given at \eqref{def:J}. If $r_0 \neq 0$ then using Lemma~\ref{lem:J} with $k=1$ and $k=2$, we have that
\begin{align*}
 \int_{\mathcal{A}_{R/2}^c}\big||x|-r_0\big|^2\omega^\eps(x,t)\dd x & \le \int_{\cal U}\big||x|-r_0\big|^2\omega^\eps(x,t)\dd x \\
 & = J_2^\eps(t) - 2 r_0 J_1^\eps(t) + r_0^2 \gamma \\
 & = \mathcal{O}\left(\frac{1}{|\ln\eps|}\right).
\end{align*}
If $r_0 = 0$, we simply use Lemma~\ref{lem:J} with $k=2$ to obtain that
\begin{align*}
\int_{\mathcal{A}_{R/2}^c}\big||x|-r_0\big|^2\omega^\eps(x,t)\dd x \le \int_{\cal U}|x|^2\omega^\eps(x,t)\dd x = J_2^\eps(t) = \mathcal{O}\left(\frac{1}{|\ln\eps|}\right).
\end{align*}
In the end we have obtained that
\begin{equation}\label{Mar3:eq1}
 \der{}{t}\mu_t(R,\eta) \le A_\eps(R,\eta) m_t(R)
\end{equation}
with
\begin{equation}\label{Mar3:eq2}
 A_\eps(R,\eta) = C \left( \frac{1}{\eta |\ln\eps|} + \frac{1}{\eta^{\alpha+1}|\ln\eps|}+\frac{1}{\eta^2R^2|\ln\eps|^2}\right).
\end{equation}
We note that the expression of $A_{\varepsilon}(R,\eta)$ coincides with the one obtained in \cite[Equation (3.44)]{Mar3} and \cite[Equation (7.3)]{Hientzsch_Lacave_Miot_2022_Dynamics_of_PV_for_the_lake_eq}. It hence suffices to reproduce verbatim the remaining part of the proof of \cite[Lemma 3.4]{Mar3} in order to complete the proof of Lemma~\ref{lemma-step}. Indeed, the rest of the proof of \cite[Lemma 3.4]{Mar3} is an iterative argument only based on \eqref{Mar3:eq1}-\eqref{Mar3:eq2}.
\end{proof}

We finish this section with the strong localization property, namely {\it(ii)} of Theorem~\ref{theo:reduit}.

\begin{prop}\label{prop:strong-loc}
For every $\kappa < \frac{1}{4}$, there exists constants $C_{\kappa,T}$ and $\eps_{\kappa,T}>0$, such that for every $\eps \in (0,\eps_{\kappa,T}]$ and for every $t \in [0,T_\eps]$,
 \begin{equation*}
 \supp \omega^\eps(\cdot,t) \subset \left\{ x \in \R^2 \, , \, \Big||x| - |z|\Big| \le \frac{C_{\kappa,T}}{|\ln\eps|^\kappa} \right\}.
 \end{equation*}
\end{prop}

Note that for the axisymmetric 3D Euler equations without swirl, the identical localization property is shown in \cite[Equation (3.8)]{Mar3}. Once Lemma~\ref{lemma-step} is proven, being the adaptation of \cite[Lemma 3.4]{Mar3}, the proof of Proposition~\ref{prop:strong-loc} follows the same lines as the \cite[Proof of (3.8), p.70]{Mar3}. For the sake of a concise exposition, we refer to \cite{Mar3} for full details of the continuity argument, only based on the inequality in Lemma~\ref{lemma:Rt} and the limit of Lemma~\ref{lemma-step}.

 The proof of Theorem~\ref{theo:reduit} is then complete.

\section{End of the proof of Theorem~\ref{theo:main}.}\label{sec:conclusion}

We have now everything we need to conclude the proof of Theorem~\ref{theo:main}. 

We already proved in Section~\ref{sec:reduction} that for every $i \in \{1,\ldots,N\}$, we can consider $\omega_i^\eps$ as a single blob of vorticity evolving in an exterior field $F_i^\eps$ and by construction, $\omega_{i,0}^\eps$ and $F_i^\eps$ satisfy \eqref{hyp:omega0_reduced} and \eqref{hyp:F_reduced} respectively with the appropriate choice of constants. Therefore, we can apply Theorem~\ref{theo:reduit} to $\omega_i^\eps$ and since the number of blobs is finite, we obtain directly that properties $(i)$ and $(ii)$ of Theorem~\ref{theo:main} both holds true on the time interval $[0,T_\eps]$.

What remains to prove is simply that $T_\eps = T$ at least for $\eps$ small enough. This is obtained easily by contradiction. Assume that $T_\eps < T$ for some $\eps > 0$. Let $\kappa \in (0,1/4)$. Then we proved that there exists a constant $C_{\kappa,T}$ independent of $\eps$ such that
 \begin{equation*}
 \supp \omega_i^\eps(\cdot,T_\eps) \subset \left\{ x \in \R^2 \, , \, \Big||x| - |z_{i,0}|\Big| \le \frac{C_{\kappa,T}}{|\ln\eps|^\kappa} \right\}.
 \end{equation*}
Assume now that $\eps$ is small enough such that $\frac{C_{\kappa,T}}{|\ln \eps|^\kappa} < \eta_0/2$. By continuity, this is in contradiction with the definition of $T_\eps$ given at \eqref{def:T_eps}. Therefore $T_\eps = T$ and the proof of Theorem~\ref{theo:main} is now complete.

\bigskip
\noindent
{\bf Acknowledgements.} This work was supported by the BOURGEONS project, grant ANR-23-CE40-0014-01 of the French National Research Agency (ANR). M.D. was funded by the Simons Collaboration of Wave Turbulence and thanks the Institut Fourier for its hospitality. C.L. also benefited of the support of the ANR under France 2030 bearing the reference ANR-23-EXMA-004 (Complexflows project of the PEPR MathsViVEs).

\appendix

\section{Appendix}
We start by recalling a lemma of rearrangement which is Lemma B.1 of \cite{Hientzsch_Lacave_Miot_2022_Dynamics_of_PV_for_the_lake_eq}.

Let $g$ be a non increasing continuous function from $(0,+\infty)$, non-negative, such that $s \mapsto sg(s) \in L^1_{\mathrm{loc}}\big([0,\infty)\big)$. Let
\begin{equation*}
 \mathcal{E}_{M,\gamma} = \left\{ f \in L^\infty_c(\R^2) \, , \, 0 \le f \le M \, , \, \int f = \gamma \right\}.
\end{equation*}

\begin{lemme}\label{lem:rearrangement}
For all $x\in\R^2$, we have
\begin{equation*}
 \max_{f \in \mathcal{E}_{M,\gamma}} \int_{\R^2} g\big(|x-y|\big) f(y)\dd y = 2\pi M \int_0^R sg(s) \dd s
\end{equation*}
where $R = \sqrt{\frac{\gamma}{\pi M}}$, namely that $f^* = M\Ind_{B(x,R)}$ is the map that maximizes this quantity on $\mathcal{E}_{M,\gamma}$.
\end{lemme}

We now introduce a variant of Gronwall's inequality.
\begin{lemme}\label{lem:gronwall}
Let $f : \R^n \to \R^n$ such that there exists $\kappa$ such that
\begin{equation*}
 \forall \, x,y \in \R^n, \quad \big|f(x)-f(y)\big| \le \kappa|x-y|.
\end{equation*}
Let $g \in L^1(\R_+,\R_+)$ and $T \ge 0$. We assume that $z : \R_+ \to \R^n$ satisfies
\begin{equation*}
 \forall t \in [0,T], \quad z'(t) = f(z(t)),
\end{equation*}
that $y : \R_+ \to \R^n$ satisfies
\begin{equation*}
 \forall t \in [0,T], \quad |y'(t)-f(y(t))| \le g(t).
\end{equation*}
Then
\begin{equation*}
 \forall t \in [0,T], \quad |y(t) - z(t)| \le \left( \int_0^t g(s) \dd s + |y(0)-z(0)|\right)e^{\kappa t}.
\end{equation*}
\end{lemme}
\begin{proof}
We have that for all $t \in [0,T]$,
\begin{align*}
 |y(t)-z(t)| & \leq \left|\int_0^t \big(y'(s) - z'(s) \big)\dd s \right| + |y(0)-z(0)|
 \\ & \le \int_0^t g(s)\dd s + \left| \int_0^t \big(f(y(s)) - f(z(s))\big) \dd s \right| + |y(0)-z(0)|
 \\ & \le \int_0^t g(s)\dd s + |y(0)-z(0)| + \kappa \int_0^t |y(s)-z(s)|\dd s ,
\end{align*}
so using now the classical Gronwall's inequality, since $t\mapsto \int_0^t g(s)\dd s + |y(0)-z(0)|$ is non negative and differentiable, we have that
\begin{equation*}
 |y(t)-z(t)| \le \left( \int_0^t g(s) \dd s + |y(0)-z(0)|\right)e^{\kappa t}.
\end{equation*}
\end{proof}

\adrese
\end{document}